\documentclass[a4paper]{article}

\usepackage{subfigure}
\usepackage{graphicx}
\usepackage{amsmath}
\usepackage{amsfonts}
\usepackage{amssymb} 
\usepackage{amsthm}
\usepackage{float}
\usepackage{subfigure}
\usepackage{algorithm}
\usepackage{a4wide}

\newtheorem{defi}{Definition}
\newtheorem{ass}{Assumption}
\newtheorem{thm}{Theorem}
\newtheorem{lem}{Lemma}
\newtheorem{prop}{Proposition}
\newtheorem{lemma}{Lemma}
\newtheorem{exm}{Example}
\newtheorem{rem}{Remark}

\DeclareMathOperator*{\argmin}{\arg \min}
\DeclareMathOperator*{\argmax}{\arg \max}
\DeclareMathOperator{\ess}{ess}
\DeclareMathOperator*{\esup}{\ess \sup}
\DeclareMathOperator{\umc}{\mathcal{U}}

\DeclareMathOperator{\hmc}{\mathcal{H}}

\DeclareMathOperator{\rinf}{\mathbb{R} \cup \{+ \infty\}}

\DeclareMathOperator{\rpos}{\mathbb{R}_{> 0}}
\DeclareMathOperator{\dom}{dom}

\newcommand{\tv}{\text{TV}}

\renewcommand{\div}{\text{div}}
\newcommand{\bv}[1][\Omega]{\text{BV}\!\left(#1\right)}%

\newcommand{\dupr}[3][L^2(\Omega)]{\left\langle #2, #3 \right\rangle_{#1}}
\newcommand{\ltwon}[3][L^2(\Sigma)]{\frac{#2}{2} \left\| #3 \right\|^2_{#1}}
\newcommand{\norm}[1]{\Vert #1 \Vert}%
\newcommand{\R}{\mathbb{R}}%
\newcommand{\sign}{\text{sign}}
\newcommand{\realcl}[1][\mathbb{R}]{#1 \cup \{ + \infty \}}
\newcommand{\ul}{u_{\lambda}}

\numberwithin{equation}{section}

\begin{document}

\titlepage

\title{Ground States and Singular Vectors of Convex Variational Regularization Methods}
\author{Martin Benning\footnotemark[1] \and Martin Burger\footnotemark[2]  }
\date{\today}

\maketitle
\renewcommand{\thefootnote}{\fnsymbol{footnote}}
\footnotetext[1]{Magnetic Resonance Research Centre, Department of Chemical Engineering and Biotechnology, c/o Cavendish Stores, JJ Thomson Avenue, Cambridge, CB3 0HE, United Kingdom
(mb941@cam.ac.uk)}
\footnotetext[2]{Westf\"alische Wilhelms-Universit\"at M\"unster, Institut f\"ur
Numerische und Angewandte Mathematik, Einsteinstr. 62, D 48149 M\"unster, Germany
(martin.burger@wwu.de)}

\renewcommand{\thefootnote}{\arabic{footnote}}

\vspace*{-12pt}

\begin{abstract}
Singular value decomposition is the key tool in the analysis and understanding of linear regularization methods in Hilbert spaces. Besides simplifying computations it allows to provide a good understanding of properties of the forward problem compared to the prior information introduced by the regularization methods. In the last decade nonlinear variational approaches such as $\ell^1$ or total variation regularizations became quite prominent regularization techniques with certain properties being superior to standard methods. In the analysis of those, singular values and vectors did not play any role so far, for the obvious reason that these problems are nonlinear, together with the issue of defining singular values and singular vectors in the first place.

In this paper however we want to start a study of singular values and vectors for nonlinear variational regularization of linear inverse problems, with particular focus on singular one-homogeneous regularization functionals. A major role is played by the smallest singular value, which we define as the ground state of an appropriate functional combining the (semi-)norm introduced by the forward operator and the regularization functional. The optimality condition for the ground state further yields a natural generalization to higher singular values and vectors involving the subdifferential of the regularization functional, although we shall see that the Rayleigh principle may fail for higher singular values. 

Using those definitions of singular values and vectors, we shall carry over two main properties from the world of linear regularization. The first one is gaining information about scale, respectively the behavior of regularization techniques at different scales. This also leads to novel estimates at different scales,  generalizing the estimates for the coefficients in the linear singular value expansion. The second one is to provide classes of exact solutions for variational regularization methods. We will show that all singular vectors can be reconstructed up to a scalar factor by the standard Tikhonov-type regularization approach even in the presence of (small) noise. Moreover, we will show that they can even be reconstructed without any bias by the recently popularized inverse scale space method.

 
\end{abstract}

\section{Introduction}
Regularization methods and their analysis are a major topic in inverse problems and image processing. In the last century mainly linear regularization methods for problems in Hilbert spaces have been studied and analyzed, and it seems that for such methods a quite complete theory is now available based on singular value decomposition  (cf. \cite{stewart}) of the forward operator in the norm defined by the regularization (cf. \cite{EHN96}), respectively generalizations to spectral decompositions in the rare cases of non-compact forward operators (cf. \cite{EHN96}).

The research in the 21st century has significantly shifted from linear regularizations to nonlinear approaches, in particular variational methods generalizing Tikhonov regularization, where singular regularization functionals such as $\ell^1$-norms or the total variation are used. In many examples the above mentioned functionals have shown to yield improved properties with respect to the incorporation of prior knowledge and the quality of reconstructions, and have also been a key tool in the adjacent theory of compressed sensing (cf. \cite{candestao1,candestao2,donoho1,donoho2}). Various advances in the analysis of such regularization methods have been made over the last years, ranging from basic regularization properties (cf. e.g. \cite{acarvogel,chamlion}) over error estimation (cf. \cite{BO04,burgresm,resm1,resmeritascherzer,BB11,hofmann,lorenz}) to corrections of inherent bias by iterative and time-flow techniques (cf. \cite{groetsch,oshburgolxuyin,gilboa,BFOS07,YOGD08}). Singular values and vectors did so far not play any role in the analysis of such methods and it is common belief that their use is restricted to linear regularization methods. This is not surprising, since first of all it is not trivial to define a notion of singular values and to characterize it in the nonlinear case. Moreover, it is obvious that due to missing linearity no decomposition into singular values can be achieved. For these reasons the study of singular values (or eigenvalues of regularization functionals) has been mainly abandoned in the inverse problems community, studies of related nonlinear eigenvalue problems rather exist in nonlinear partial differential equations and functional inequalities (cf. \cite{agueh,berestycki,browder,dacorognagangbo,dolbeault,franzina,henrot,kawohl2,kawohl3,kawohl4,weinstein}), in control theory (cf. \cite{fujimoto}), in image processing (cf. \cite{alter2,alter1,bellettini,caselles}) and surprisingly in machine learning (cf. \cite{hein,szlam}). Motivated by those as well as general approaches to nonlinear eigenvalue problems we will define a ground state by a Rayleigh-type principle and further singular values and singular vectors by considering the first-order optimality condition for the non-convex variational problem defining ground states. The main results we shall derive are the following:
\begin{itemize}

\item First of all, our definition of singular values and singular vectors is studied and demonstrated to be a meaningful extension of the linear case, although some properties can be lost in extreme cases, e.g. the discreteness of the spectrum and the Rayleigh principle for higher singular values (Section 3). 

\item With the singular values and singular vectors we derive error estimates for appropriate linear functionals of the solution, which provide information about the behavior at different scales. This is made explicit for an example in total variation denoising (Section 4).

\item An important part is to verify that singular vectors are exact solutions of variational regularization schemes. This means that if the image of a singular vector under the forward operator is used for the reconstruction, the solution is a multiple of the singular vector. Surprisingly, under certain conditions particularly met for singular regularizations, the same holds true if a certain amount of noise is added. For inverse scale space methods we can further verify that singular vectors are reconstructed without bias, i.e. after finite time (depending on the singular value) the solution of the inverse scale space equals exactly the singular vector, without a multiplicative change (Sections 5 and 7).
%
%

\item We derive estimates on the bias of variational regularization schemes, which show that the minimal bias is somehow defined by the ground state, respectively by the smallest singular value in our definition (Section 6). 

\item We provide a variety of examples of inverse problems and regularization functionals, for which singular values and singular vectors can be computed explicitly. This allows to draw various conclusions about the behavior of the regularization and the typical shape of preferred solutions (Section 8).

\end{itemize}

\section{Notations and Assumptions}

To fix notation, we consider linear inverse problems of the form
\begin{align}
Ku = f \, \text{,}\label{eq:invprob}
\end{align} 
with $K:\umc \rightarrow \hmc$ being a linear operator mapping from a Banach
space $\umc$ to a Hilbert space $\hmc$, with the goal to recover $u$
from \eqref{eq:invprob} with given data $f$ that is potentially being corrupted by noise. 
%
We are mainly interested in the case of $K$ being compact, in particular continuous from the weak or weak-* topology of $\umc$ to the strong topology of $\hmc$, which creates ill-posedness of the inverse problems.

Nonlinear variational regularization methods for computing robust approximate solutions of \eqref{eq:invprob}
are of the form
\begin{align}
\hat{u} \in \argmin_{u \in \dom(J)} \left\{ \frac{1}{2} \left\| Ku - f
\right\|^2_{\hmc} + \alpha J(u) \right\} \, \text{,} \label{eq:varframe}
\end{align}
with $J:\dom(J) \subseteq \umc \rightarrow \realcl$ being a so-called
regularization functional that incorporates the a-priori knowledge, and $\alpha \in \R_{> 0}$ denoting the regularization parameter that controls the impact of $J$ on the solution $\hat{u}$ of \eqref{eq:varframe}. 
Note that in the variational approach linear regularizations methods are related to quadratic regularization functionals like
$$J(u) = \frac{1}{2} \| Du \|^2_{\umc} $$
for linear operators $D:\umc \rightarrow \umc$, since they lead to linear optimality conditions. 
Another classical choice motivated from statistical mechanics and information theory is the Boltzmann (Shannon) entropy regularization functional, in $\umc=L^1(\Omega)$,  
$$ J(u) = \int_\Omega u \log(u) - u~dx, $$
leading to the so-called maximum entropy regularization (cf. \cite{eggermont}).
Recently popular functionals are non-differentiable regularization energies like the one-norm
$J(u) = \|u\|_{\ell^1}$ in $\umc=\R^N$ or the total variation $J(u) = \tv(u)$,  being defined as
\begin{align}
\tv(u) := \sup_{\substack{\varphi \in C_0^\infty(\Omega; \R^n) \\ \|
\varphi \|_{L^\infty(\Omega; \R^n)} \leq 1}} \int_\Omega u ~\div \varphi ~dx  \,
\text{.}\label{eq:totalvar}
\end{align}
Total variation regularization became popular in  the Rudin-Osher-Fatemi (ROF) model \cite{ROF}
\begin{align}
\hat{u} \in \argmin_{u \in \bv} \left\{ \frac{1}{2} \left\| u - f
\right\|^2_{L^2(\Omega)} + \alpha \tv(u) \right\} \, \text{.} \label{eq:rof}
\end{align} 
The space $\bv$ is the space of all function $u \in L^1(\Omega)$
such that $\tv(u)$ is bounded. In case of $\Omega \subseteq
\R^n$ for $n \in \{1, 2\}$ the space $\bv$ can be embedded into $L^2(\Omega)$.\\

Since we are going to deal with rather large classes of convex functionals $J$, let us recall some basic facts from convex analysis (cf. \cite{rockafellar,ekelandtemam} for detailed discussions).
As usual for a Banach space $\umc$, the Banach space of bounded linear mappings from $\umc$ to $\R$
is called the dual space of $\umc$ and is denoted by $\umc^{*}$, with norm
\begin{align*}
\|p\|_{\umc^{*}} := \sup_{\|u\|_{\umc} = 1} | p(u) | = \sup_{u \in
\umc\setminus\{0\}} \frac{|p(u)|}{\|u\|_{\umc}} =
\sup_{\|u\|_{\umc} \leq 1} | p(u) | \, \text{.}
\end{align*}
The functional $p(u) = \dupr[\umc^{*} \times \umc]{p}{u}$ is called the dual product.
Throughout this work we are going to denote the dual product simply by
$\dupr[\umc]{p}{u}$. In case that $\umc$ is even a Hilbert space, the dual
product can be identified with the scalar product of $\umc$.

The characterization of dual spaces and its elements allows us to define the
subdifferential of a convex functional.
\begin{defi}[Subdifferential]
Let $\umc$ be a Banach space with dual space $\umc^{*}$, and let the proper
functional $J:\umc \rightarrow \rinf$ be convex. Then, $J$ is called
subdifferentiable at $u \in \umc$, if there exists an element $p \in
\umc^{*}$ such that
\begin{align*}
J(v) - J(u) - \langle p, v - u \rangle_{\umc} \geq 0
\end{align*}
holds, for all $v \in \umc$. Furthermore, we call $p$ a subgradient at position
$u$. The collection of all subgradients at position $u$, i.e.
\begin{align*}
\partial J(u) := \left\{ p \in \umc^{*} \ | \ J(v) - J(u) - \langle p, v - u
\rangle_{\umc} \geq 0 \, \text{,} \, \forall v \in \umc \right\} \subset
\umc^{*} \, \text{,}
\end{align*}
is called subdifferential of $J$ at $u$.
\end{defi}

We further mention that the subdifferential of one-homogeneous functionals can be further characterized as
\begin{align}
\partial J(u) := \left\{ p \in \umc^{*} \ | \ \langle p, u
\rangle = J(u), \ \langle p, v
\rangle \leq J(v) \, \text{,} \, \forall v \in \umc \right\} \subset
\umc^{*} \, \text{.}\label{eq:onehomosubdiff}
\end{align}

Another concept we shall use in several arguments is the notion of (generalized) Bregman distances, defined as
\begin{align*}
D_J^p(v,u) = J(v) - J(u) - \dupr[\umc]{p}{v-u} \, \text{,}
\end{align*}
with $p \in \partial J(u)$. Bregman distances are not common distance functionals, since they do not satisfy a triangle inequality and are not symmetric in general. However, for $J$ being convex they are non-negative and satisfy $D_J^p(u,u)=0$. Symmetry can be restored by using symmetric Bregman distances, i.e. 
\begin{align*}
D_J^{p,q,\text{symm}}(v,u) =D_J^{p}(u, v) + D_J^{q}(v, u) =\dupr[\umc]{q-p}{v-u}
\end{align*}
for $q \in \partial J(v)$. 

Before we continue with the definition of ground states and singular vectors for
general convex and subdifferentiable regularization functionals, we want to
precisely define the class of operators and functionals we are going to
investigate. Thus, for the remainder of this work we will assume the following properties without further notice:
\begin{ass}[Setup]\label{defi:setup}\hspace{\linewidth}
\begin{itemize}
  \item $\Omega \subseteq \R^d$, $\Sigma \subseteq \R^k$ are bounded domains.
  \item $\umc$ is a Banach space, being the dual of some other Banach space.
  \item $\hmc$ is a Hilbert space.
  \item $K:\umc \rightarrow \hmc$ is a bounded linear operator mapping between
these spaces
\item $J:\dom(J) \subseteq \umc \rightarrow \rinf$ is a proper non-negative convex functional 
\end{itemize}
\end{ass} 

\section{Ground States and Singular Vectors}\label{sec:gssv}

In this section we want to define ground states of regularization functionals as well as an analogue of singular vectors for nonlinear functionals.

\subsection{Ground States}

We start with a definition of a ground state, which is motivated by similar properties in partial differential equations, e.g. the classical one of a ground-state in the Schr\"odinger equation and related problems (cf. \cite{berestycki,weinstein,agueh}). In order to 
obtain a ground state we normalize the element $u$ and minimize the regularization functional among those elements, i.e. $u_0$ is defined as
	\begin{align}
		u_0 \in  \argmin_{\substack{u \in \dom(J) \\ \| Ku
		\|_{\hmc} = 1}} \left\{ J(u) \right\} \,
		\text{.}\label{eq:groundsttrivial}
	\end{align}
In the context of variational schemes like \eqref{eq:varframe} we are
particularly interested in non-trivial ground states of one-homogeneous regularization
functional. A trivial ground state appears if $J(u_0) = 0$, and we can immediately provide a well-known example for such:

\begin{exm}\normalfont
Let $\Omega \subset \mathbb{R}^d$ with $d \in \{1, 2\}$. Then
we know that $\bv \subset L^2(\Omega)$ holds. Thus, for $K = I$ being
the identity operator $I:L^2(\Omega) \rightarrow L^2(\Omega)$, a trivial
ground state of $J = \tv$ is the constant function $u_0 = 1/\sqrt{|\Omega|}$, since we
have $\tv(u_0) = 0$ and $\| u_0 \|_{L^2(\Omega)} = 1$. Note that as usual the ground
state is not unique, since $-u_0$ is a ground state as well.
\end{exm}
However, in many cases trivial ground states do not give interesting insights
into the nature of a regularization energy, as the previous example shows. Thus,
we would like to investigate non-trivial ground states that are orthogonal to
the trivial ones in a reasonable sense. Let us therefore define some preliminary
notions first.
The $K$-product of two elements $u, v \in \umc$ is defined
as
\begin{align*}
\langle u, v \rangle_K := \langle Ku, Kv
\rangle_{\hmc} \, \text{.}
\end{align*}
Furthermore we are going to write $\|u\|_K$ as an abbreviation for $\sqrt{\langle u, u \rangle_K}$. This particular definition of a scalar product for elements of a Banach space
allows us to define a useful orthogonal complement of a kernel of a regularization
functional, which we define as usual via
\begin{align}
\ker(J)^{\bot} := \{ u \in \dom(J) \ | \ \langle u, v \rangle_K = 0,
\forall v \in \ker(J) \} \, \text{.}
\end{align}
For completeness we also introduce
\begin{align*}
\ker(J) := \{ u \in \dom(J) \ | \ J(u) = 0 \} \, \text{.}
\end{align*}

In the case of convex non-negative one-homogeneous functionals we are mainly interested in, the kernel and also its complement can further be characterized as linear subspaces:

\begin{lemma}
Let $J$ be convex, non-negative and one-homogeneous. Then $\ker(J)$ is a linear subspace.
\end{lemma}
\begin{proof}
Let $u,v \in \ker(J)$ and $a,b \in \R$ such that $|a|+|b|\neq 0$. With $\alpha=\frac{|a|}{|a|+|b|} \in [0,1]$, $\tilde u= \sign(a) u$ and
$\tilde v = \sign(b) v$ we obtain
\begin{align*}
\begin{split}
 0 &\leq  J(au+bv) = J((|a|+|b|)(\alpha \tilde u + (1-\alpha) \tilde v))  \\
&= (|a|+|b|) J(\alpha \tilde u + (1-\alpha) \tilde v) \leq (|a|+|b|) (\alpha J(\tilde u) + (1-\alpha) J(\tilde v)),
\end{split} \,
\end{align*}
by using the one-homogeneity and convexity of $J$. Since $J(\tilde u) = J(u) =0$ and $J(\tilde v) = J(v) =0$ hold due to the one-homogeneity of $J$, we conclude $J(au+bv) = 0$ as well. Thus, $au+bv \in \ker(J)$ holds true.
\end{proof}

\begin{exm}\normalfont\label{exm:kernel}
Considering $J = \tv$ again, we easily see that $\ker(\tv)$ equals the
set of all constant functions, because the estimation of the kernel can simply
be reduced to estimating $\ker(\nabla)$.
\end{exm}
\begin{figure}[!ht]
\begin{center}
\includegraphics[scale=0.4]{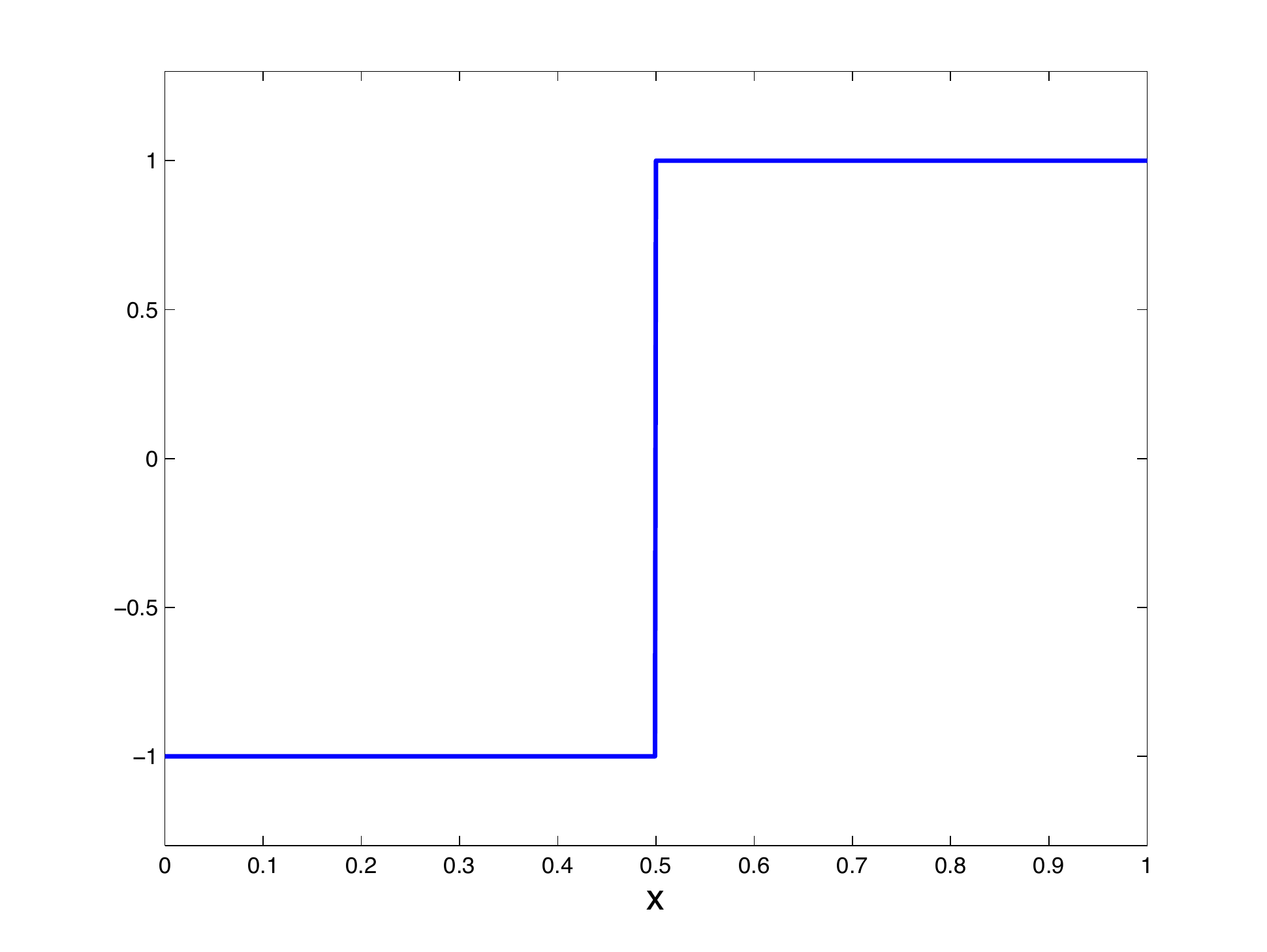}
\end{center}
\caption{The function $u^a$ as defined in \eqref{uadefinition}, for $a = 1/2$. This function is a ground state of $K=I$, $J=\tv$, according to Definition \ref{defi:gs}.}
\label{fig:1dtvgs}
\end{figure}
\begin{defi}[Ground State]\label{defi:gs}
Under the above assumptions on $J$ and $K$, a  ground state $u_0$ is defined as an element
	\begin{align}
		u_0 \in \argmin_{\substack{u \in \ker(J)^{\bot} \\ \| Ku
		\|_{\hmc} = 1}} \left\{ J(u) \right\} \,
		\text{.}\label{eq:groundst}
	\end{align}
Moreover, if $u_0$ exists we call 
\begin{equation}
	\lambda_0 = J(u_0) 
\end{equation}
the smallest singular value. 
\end{defi}
Under standard assumptions on variational regularization methods, the ground state indeed exists, which can be verified by usual arguments:
\begin{thm}
Let $d$ be a metric on $\umc$, let $K$ be continuous from this metric topology to the strong topology of $\hmc$, and let $J$ be lower semicontinuous with respect to this metric topology. Moreover, let the sublevel sets of $u \mapsto \norm{Ku}_{\hmc} + J(u)$ be compact in the metric topology. Then there exists at least one ground state $u_0 \in \umc$.
\end{thm}
\noindent We want to give a brief example of ground states in case of $J(u) = \tv(u)$.
\begin{exm}\label{exm:gs}\normalfont
For $\Omega = [0, 1] \subset \mathbb{R}$, $K = I$ and $J(u) = \tv(u)$ we want to consider
\begin{equation}
u^a(x) = \begin{cases} \frac{\sqrt{(1 - a)a}}{1-a} & x \geq a\\ -\frac{\sqrt{(1 - a)a}}{a} & x < a \end{cases} \, \text{,}	\label{uadefinition}
\end{equation}
for $a \in (0, 1)$. We easily see that $u^a$ is orthogonal to the kernel of
$\tv$, which consists of all constant functions due to Example \ref{exm:kernel},
since $\int_0^1 u^a(x) ~dx = 0$ holds. Moreover, $u^a$ guarantees the normalization constraint $\| u^a \|_{L^2([0, 1])} = 1$ for every $a \in ]0, 1[$. However, the $\tv$-value $\tv(u^a) = \frac{1}{\sqrt{(1 - a)a}}$ is a strictly convex function in $a$ with unique minimum at $a = 1/2$. Thus, $u_0 = u^{\frac{1}{2}}$, which is visualized in Figure \ref{fig:1dtvgs}, is a ground state if we can prove that there does not exist a function $\tilde{u}$ with $\dupr[{L^2([0, 1])}]{\tilde{u}}{1} = 0$, $\|\tilde{u}\|_{L^2([0, 1])} = 1$ and $\tv(\tilde{u}) < 2$.
\begin{lem}\label{lem:tv1Dgs}
There exists no function $\tilde{u} \in C([0, 1])$ with $\|\tilde{u}\|_{L^2([0, 1])} = 1$ and $\langle 1, \tilde{u} \rangle_{L^2([0, 1])} = 0$ such that $\tv(\tilde{u}) < 2$ holds.
\begin{proof}
It is easy to see that for the monotonic rearrangement of an arbitrary function $\tilde{u} \in C([0, 1])$, which we want to denote by $\tilde{u}^*$, we have (cf. e.g \cite{agueh})
\begin{align*}
\tv(\tilde{u}^*) \leq \tv(\tilde{u}) \, \text{.}
\end{align*}
Thus, in the following we are going to consider monotonically increasing functions $\tilde{u}^* \in C([0, 1])$ with $\|\tilde{u}^*\|_{L^2([0, 1])} = 1$ and $\langle 1, \tilde{u}^* \rangle_{L^2([0, 1])} = 0$ only, without loss of generality. Now we want to prove $\tv(\tilde{u}^*) \geq 2 $ by contradiction and therefore subdivide the proof into two parts. First of all we are going to prove the inequalities 
\begin{align}
(\tilde{u}^*(1))^2 - 1 &\geq \left( (\tilde{u}^*(1))^2 - (\tilde{u}^*(0))^2\right)y\label{eq:gsproof1}
\intertext{and}
1 &\leq \tv(\tilde{u}^*) \left( - \int_0^y \tilde{u}^*(x)  \right) \, \text{,}\label{eq:gsproof2}
\end{align}
with $y$ denoting a root of $\tilde{u}^*$. Subsequently we are going to use \eqref{eq:gsproof1} and \eqref{eq:gsproof2} to conclude the contradiction.

For a monotonically increasing function $\tilde{u}^*$ we can rewrite the normalization constraint $\| \tilde{u}^* \|_{L^2([0, 1])} = 1$ to
\begin{align*}
1 = \| \tilde{u}^* \|^2_{L^2([0, 1])} &= \int_0^y \left(\tilde{u}^*(x)\right)^2
~dx + \int_y^1 \left(\tilde{u}^*(x)\right)^2 ~dx\\
&\leq y \left(\tilde{u}^*(0)\right)^2 + (1 - y) \left(\tilde{u}^*(1)\right)^2 \, \text{,}
\end{align*}
with $y$ denoting a root of $\tilde{u}^*$. Rearranging immediately yields \eqref{eq:gsproof1}.

Moreover, according to the second mean value theorem of integration there exists a $\xi \in ]0, 1[$ such that we can rewrite the normalization constraint to
\begin{align*}
1 &= \int_0^1 \tilde{u}^*(x) \tilde{u}^*(x) ~dx\\
&= \tilde{u}^*(0) \int_0^\xi \tilde{u}^*(x) ~dx + \tilde{u}^*(1) \int_\xi^1 \tilde{u}^*(x) ~dx\\
&= \underbrace{\left( \tilde{u}^*(1) - \tilde{u}^*(0)\right)}_{= \tv(\tilde{u}^*)} \int_\xi^1 \tilde{u}^*(x) ~dx \, \text{,}
\end{align*}
where the last equality holds due to $\langle 1, \tilde{u}^* \rangle_{L^2([0, 1])} = 0$. Since the value of $\int_\xi^1 \tilde{u}^*(x)~dx$ gets maximal for $\xi = y$ and since we know $\int_y^1 \tilde{u}^*(x) ~dx = - \int_0^y \tilde{u}^*(x) ~dx$ we obtain \eqref{eq:gsproof2}.

Finally, we are now able to proof the lemma's statement via contradiction. We
assume $\tilde{u}^*$ to satisfy $\tv(\tilde{u}^*) < 2$. Then, the normalization
constraint $\| \tilde{u}^* \|_{L^2([0, 1])} = 1$ however implies either $\tilde{u}^*(1) \geq 1$ or $\tilde{u}^*(0) \leq -1$. Without loss of generality we assume $\tilde{u}^*(1) \geq 1$. Thus, there exists a constant $c \geq 0$ such that $\tilde{u}^*(1) = 1 + c$ holds. Due to $\tv(\tilde{u}^*) < 2$ this automatically implies $\tilde{u}^*(0) \geq -1 + c$ and $(\tilde{u}^*(0))^2 \leq (-1 + c)^2$. Applying \eqref{eq:gsproof1} therefore yields
\begin{align*}
(1 + c)^2 - 1 \geq \left( (1 + c)^2 - (1 - c)^2\right)y \, \text{,}
\end{align*}
which can be rewritten to $y \leq c/4 + \frac{1}2$. We are therefore able to estimate \eqref{eq:gsproof2} to obtain
\begin{align*}
1 &\leq \tv(\tilde{u}^*) \left( - \int_0^y \tilde{u}^*(x) ~dx \right) \\
&\leq \tv(\tilde{u}^*) y \left(-\tilde{u}^*(0)\right) \leq \tv(\tilde{u}^*) y (1 - c)\\
&\leq \tv(\tilde{u}^*) \left( \frac{c}{4} + \frac{1}{2} \right) (1 - c) \\
&\leq \tv(\tilde{u}^*) \left(\frac{1}{2} - \frac{c^2}{4} - \frac{c}{4} \right) \leq \frac{1}{2} \tv(\tilde{u}^*) \, \text{,}
\end{align*}
which yields $2 \leq \tv(\tilde{u}^*)$ and therefore is a contradiction to the assumption $\tv(\tilde{u}^*) < 2$.
\end{proof}
\end{lem}
\end{exm}
Note that $u_0$ of the previous example is not a unique ground state, since $-u_0$ yields the same minimum. However, this is also true for ground states of quadratic variational schemes and therefore no surprise. The following example shows that ground states are not unique in general.
\begin{exm}\label{exm:ell1gs}\normalfont
Let $\umc={\ell^1(\R^N)}$ and $\hmc = \ell^2(\R^M)$.
Consider the regularization energy $J(u) = \| u \|_{\ell^1}$, and as an operator a matrix $K$ with columns being normalized with respect to the $\ell^2$-norm, i.e. $\left(Ke_j\right)^{T} \cdot \left(Ke_j\right) = 1$, with $e_j$ denoting the $j$-th unit vector. Then every vector 
$ e_i = (\delta_{ji})_{j=1,\ldots,N}$, with $\delta_{ij}$ denoting the Kronecker delta, is a ground state.
\end{exm}

In some cases it is useful to have an alternative definition of ground states of variational regularizations like \eqref{eq:varframe}:
\begin{defi}[Ground State II]\label{defi:gs2}
A ground state
$u_0$ is defined as
	\begin{align}
		u_0 \in \argmax_{\substack{u \in \ker(J)^{\bot} \\ J(u) \leq 1}} \left\{
		\|Ku\|_{\hmc} \right\} \, \text{.}\label{eq:groundst2}
	\end{align}
If $u_0$ exists, we shall call 
\begin{equation}
	\lambda_0 = \frac{1}{\|Ku_0\|_{\hmc} }
\end{equation}
the smallest singular value. 
\end{defi}

At first glance, Definition \ref{defi:gs2} appears to be very different from Definition \ref{defi:gs}; however, for one-homogeneous functionals $J$ both definitions are equivalent up to normalization of the singular vector and yield the same singular value, as we will see with the following Lemma.
\begin{lem}
Let $J$ be a proper and one-homogeneous functional. Then, Definition \ref{defi:gs} and Definition \ref{defi:gs2} are equivalent up to multiplication of $u_0$ by  $\lambda_0$.
\begin{proof}
$\Rightarrow$: Let $u_0$ be a ground state that satisfies Definition
\ref{defi:gs}. Then, for $\tilde{u} := u_0/J(u_0) = u_0/\lambda_0$ we obtain
$J(\tilde{u}) = 1$ and $\| K\tilde{u} \|_{\hmc} = 1/J(u_0) = 1/\lambda_0$. Thus,
in order to satisfy Definition \ref{defi:gs2}, $\tilde{u}$ is supposed to
maximize $\| Ku\|_{\hmc}$, i.e. $\| K\tilde{u}\|_{\hmc} \geq \|
Kv\|_{\hmc}$ for all $v$ with $J(v) \leq 1$. We prove this statement by
contradiction and assume that there exists a function $v$ with $\| Kv \|_{\hmc}
> 1/\lambda_0$ and $J(v) \leq 1$. However, if such a function exists, we can
define $\tilde{v} := v/\|Kv\|_{\hmc}$ to obtain a function that satisfies
$\|K\tilde{v}\|_{\hmc} = 1$ and $J(\tilde{v}) \leq 1/\|Kv\|_{\hmc} <
\lambda_0$, which is a contradiction to $u_0$ being a ground state in the sense of Definition \ref{defi:gs}.\\
$\Leftarrow$: Now let $u_0$ be a ground state in terms of Definition
\ref{defi:gs2}, i.e. $J(u_0) \leq 1$ such that $\|Ku_0\|_{\hmc} =: 1/\lambda_0$
is maximized. Then we can define $\tilde{u} = \lambda_0 u_0$, which satisfies
$\|K\tilde{u}\|_{\hmc} = 1$ and $J(\tilde{u}) = \lambda_0$. In analogy to the
first part of the proof, we prove by contradiction that $\tilde{u}$ already has
to be a ground state in terms of Definition \ref{defi:gs}. We therefore assume
that there exists a function $v$ such that $J(v) < \lambda_0$ and $\|Kv\|_{\hmc}
= 1$ holds. If such a function exists, than $\tilde{v} := v/J(v)$ exists as
well. However, for $\tilde{v}$ we observe $J(\tilde{v}) = 1$ and
$\|K\tilde{v}\|_{\hmc} > 1/\lambda_0$, which is a contradiction to $u_0$ being a
ground state in terms of Definition \ref{defi:gs2}.
\end{proof}
\end{lem}

\subsection{Singular Vectors}

In analogy to singular vectors of linear operators we want to extend the concept of singular vectors to variational frameworks of the form \eqref{eq:varframe}. The motivation is to consider the formal optimality condition for the ground state, which is obtained by considering stationary points of the Lagrange functional (with parameter $\lambda \in \R$)
\begin{equation}
	L(u;\lambda) = J(u) - \frac{ \lambda}2 ( \|Ku\|_{\hmc}^2 - 1), 
\end{equation}
which is given by 
$$ \lambda K^*K \in \partial J(u), $$
where $\partial J$ denotes the subdifferential.

\noindent Hence, we define a singular vector as follows:
\begin{defi}[Singular Vector]\label{def:singvec}
Let $J$ be 
convex with non-empty subdifferential $\partial J$ at every $u \in \dom(J)$. Then, every function $\ul \neq 0$ with $\|K\ul\|_{\hmc} =
1$ satisfying
\begin{align}
\lambda K^{*}K\ul \in \partial J\left(\ul\right)\label{eq:svdefi}
\end{align}
is called singular vector of $J$ with corresponding singular value $\lambda$. The subgradient 
$p_\lambda = \lambda K^{*}K\ul$ is called dual singular vector in the following. 
\end{defi}


\noindent We mention that taking a dual product with $u_\lambda$ yields the singular value relation
\begin{align*}
\lambda &= \frac{\langle p_\lambda, \ul \rangle_{\umc}}{\|K\ul\|^2_{\hmc}} =
\langle p_\lambda, \ul \rangle_{\umc} \, \text{.}
\end{align*}
In the case of $J$ being one-homogeneous we even have
\begin{align*}
\lambda &= J(u) \, \text{,}
\end{align*}
which also implies $\lambda \geq \lambda_0$, for any singular value $\lambda$.

For smooth $J$ one can prove that the ground state is a singular vector by analyzing the Lagrange functional above (cf. \cite{hinzepinauulbrichulbrich}). In the one-homogeneous case we give an alternative proof:

\begin{prop}
Let $J$ be one-homogeneous and let $u_0$ be the ground state with $\lambda_0=J(u_0)$. Then $\lambda_0$ is a singular value and $u_0$ is a singular vector.  
\end{prop}
\begin{proof}
For $J$ being one-homogeneous, $p_0 \in \partial J(u_0)$ is equivalent to $\langle p_0, u_0 \rangle_{\umc} = J(u_0)$ and 
$$ \langle p_0, u \rangle_{\umc} \leq J(u), \qquad \forall~u \in \umc \, \text{,} $$
due to \eqref{eq:onehomosubdiff}. We verify this property for $p_0 = \lambda_0 K^* K u_0$. First of all we obtain 
$$ \langle p_0, u_0 \rangle_{\umc} = \lambda_0 \langle Ku_0, Ku_0 \rangle_{\hmc} = J(u_0) $$
by the definition of $\lambda_0$ and the normalization of $u_0$. Moreover, for arbitrary $u \in \umc$ with $Ku \neq 0$ we define  $v=u/\Vert K u \Vert_{\hmc}$ and find
$$ \langle p_0, u \rangle_{\umc} = \Vert K u \Vert_{\hmc} \lambda_0 \langle K u_0, K v \rangle_{\hmc} \leq \Vert K u \Vert_{\hmc} \lambda_0. $$
Since $v$ is normalized, we have by the definition of the ground state 
$$  \lambda_0 = J(u_0) \leq J(v) = \frac{J(u)}{\Vert Ku\Vert_{\hmc}} $$
and thus, $\langle p_0, u \rangle_{\umc} \leq J(u)$.
If $Ku=0$, then 
$$\langle p_0, u \rangle_{\umc} = \lambda_0 \langle K u_0 , Ku \rangle_{\hmc} = 0 \leq J(u), $$
thus $p_0 \in \partial J(u_0)$.
\end{proof}

Higher singular values and singular vectors are difficult to characterize as we shall see from examples below. Also orthogonality of singular vectors corresponding to different singular values is lost. Consider 
$\lambda K^* K u_\lambda = p_\lambda$ and $\mu K^* K u_\mu = p_\mu$ for singular vectors $u_\lambda$ and $u_\mu$, then we only have 
$$ \frac{1}\lambda \langle p_\lambda, u_\mu \rangle = \frac{1}\mu \langle p_\mu, u_\lambda \rangle \, \text{.} $$

\noindent We want to give two examples:
\begin{exm}\normalfont
Let us investigate the functional $J(u) = \frac{1}2 \| \nabla u \|^2_{L^2(\Omega;\mathbb{R}^n)}$, in order to demonstrate that the definition of singular vectors is consistent with the definition of singular vectors of linear operators. Since $J$ is Fr\'{e}chet-differentiable, its subdifferential consists of its Fr\'{e}chet-derivative only, i.e. $\partial J(u) = \{ -\Delta u \}$. Considering $K = I$ for simplicity, \eqref{eq:svdefi} reads as the classical Eigenfunction problem of the Laplace operator, i.e.
\begin{align*}
-\lambda u_\lambda = \Delta u_\lambda \, \text{,}
\end{align*}
or equivalently as the singular vector decomposition problem of the gradient operator $\nabla$.
Due to the compactness of the inverse Laplacian, the spectrum is discrete, i.e. there are countably many different singular values.
\end{exm}

\begin{figure}[!ht]
\begin{center}
\includegraphics[scale=0.4]{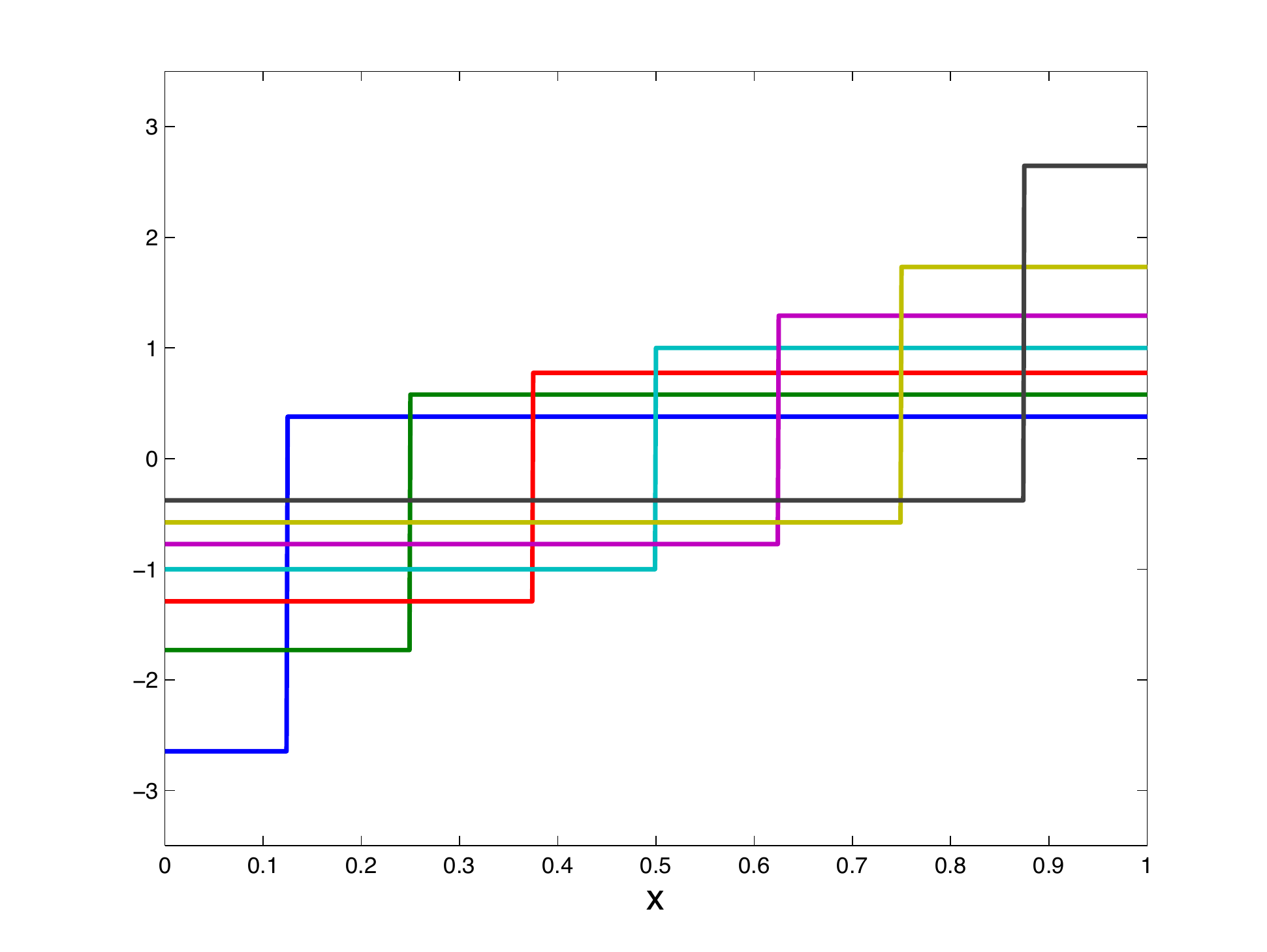}
\end{center}
\caption{The functions $u^a$ plotted for the values $a \in \{1/8, 1/4, 3/8, 1/2, 5/8, 3/4, 7/8\}$. All of these functions are singular values according to Definition \ref{def:singvec}, but only $u^{1/2}$ is a ground state.}
\label{fig:1dtvsv}
\end{figure}

The structure of the spectrum, by which we formally denote the set of singular values, changes if we consider more degenerate cases as the total variation frequently used in inverse problems and imaging:

\begin{exm}\label{exm:singvec}\normalfont
Let us now consider $K$ being the embedding operator from $\bv$ to $L^2(\Omega)$ on the unit interval
$\Omega = [0,1]$. For $a\in (0,1)$  the function $u^a$ defined by \eqref{uadefinition} is a singular vector of $\tv$, with singular value $\lambda = 1/\sqrt{(1 - a)a}$, as we shall see in the following:
We can characterize the subdifferential of $\tv$ as
\begin{align}
\partial \tv(u) = \left\{ \div \varphi ~ \left| ~ \left\| \varphi \right\|_{L^\infty(\Omega;\mathbb{R}^n)} \leq 1, \left. \varphi \cdot n \right|_{\partial \Omega}=0, \langle \div \varphi, u \rangle_{L^2(\Omega)} = \tv(u) \right. \right\} \, \text{,}
\end{align}
for $u \in \bv$. We see that the distributional derivative of the continuous function $q^a:[0, 1] \rightarrow [-1, 0]$ defined as
\begin{align*}
q^a(x) := \begin{cases} \frac{x - 1}{1 - a} & x \geq a\\ - \frac{x}{a} & x < a \end{cases} \, \text{,}
\end{align*}
for $a \in ]0, 1[$, is an element of $\partial\tv(u^a)$, since we have $\| q^a
\|_{L^\infty([0, 1])} = 1$, $q^a(0) = q^a(1) = 0$ and $\langle
\left(q^a\right)^\prime, u^a \rangle_{L^2([0, 1])} = \tv(u^a) = \lambda$ (here
the derivative has to be considered in a distributional way). Moreover, the
distributional derivative of $q^a$ satisfies the singular vector relation $\lambda u^a = \left(q^a\right)^\prime \in \partial\tv(u^a)$, for $\lambda = 1/\sqrt{(1 - a)a}$. Hence, $u^a$ is a singular vector of $\tv$. The singular vector is visualized for various choices of $a$ in Figure \ref{fig:1dtvsv}
Now we see that there exists a continuous spectrum $\left[2,\infty\right)$ although in spatial dimension one the embedding operator is compact. 
\end{exm}
\noindent We are going to consider various further examples in
Section \ref{sec:exm}.

\subsection{Failure of the Rayleigh Principle}

Above we have defined the ground state by a Rayleigh principle, i.e. minimizing $J$ with respect to normalization and orthogonality to the kernel. Again in the linear case, further singular vectors can be obtained by similar Rayleigh principles, e.g. by minimizing $J(u)$ subject to $u \in \ker(J)^\bot$, $\Vert u \Vert_K = 1$ and $\langle u, u_0 \rangle_K = 0. $ We have already seen from Example \ref{exm:singvec} that in the nonlinear case we will not be able to compute all singular values and singular vectors this way. However, it seems at least interesting whether we can obtain an orthonormal basis of singular vectors (with orthogonality defined in the $K$-scalar product, assuming for the moment that $K$ has trivial nullspace). As we shall demonstrate in the following, this is not possible in general. 

Let us first discuss formally why the Rayleigh principle can fail to yield singular vectors in the nonlinear case. 
Assume for this sake we have a system of orthonormal singular vectors $u_j$, $j=0,\ldots,n$, i.e.
$$ \lambda_j K^*Ku_{\lambda_j} = p_{\lambda_j} \in \partial J(u_{\lambda_j}), \qquad \langle u_{\lambda_j}, u_{\lambda_k} \rangle_K = \delta_{jk} . $$
Then we define
\begin{equation}
	u_{\lambda_{n + 1}} \in \argmin_{u \in {\cal C}_n} J(u),
\end{equation}
with the constraint set 
\begin{equation}
	{\cal C}_n := \left\{~u \in \umc~\left|~\Vert u \Vert_K = 1, \langle u_{\lambda_j}, u \rangle_K = 0, j=0,\ldots,n~\right.\right\}.
\end{equation}
Note that the existence of $u_{\lambda_{n + 1}}$ follows under the same conditions as the existence of a ground state if in addition ${\cal C}_n$ is nonempty (which is always the case in infinite dimensions).
Again we may set up the Lagrange functional
\begin{equation}
	L = J(u) - \frac{\lambda}2 \left( \Vert u \Vert_K^2 - 1\right) - \sum_{j=0}^n \mu_j \langle u_{\lambda_j}, u \rangle_K
\end{equation}
and thus obtain the optimality condition as the variation with respect to $u$ via 
\begin{equation}
	 \lambda K^* K u_{\lambda_{n + 1}} + \sum_{j=0}^n \mu_j K^* K u_{\lambda_j} = p_{\lambda_{n + 1}} \in \partial J(u_{\lambda_{n + 1}}). 
\end{equation}
Now $u_{\lambda_{n + 1}}$ is a singular vector if and only if $\mu_j=0$ for $j=1,\ldots,n$. To understand the latter we take a product with $u_{\lambda_{k}}$, $k \in \{0,\ldots,n\}$, and obtain
$$ \mu_k = \langle p_{\lambda_{n + 1}}, u_{\lambda_k} \rangle_{\umc}, $$ 
due to the orthonormality. Now there is no particular reason why $\langle p_{\lambda_{n + 1}}, u_{\lambda_k} \rangle_{\umc} = 0$ should hold, since we cannot use the usual argument as in the linear case, namely a simple eigenvalue equation for $u_{\lambda_k}$. 

We finish this section with a simple example that explicitly shows the failure of the Rayleigh principle:
\begin{exm}\normalfont
Let $\umc=\ell^1(\R^2)$ and $\hmc = \ell^2(\R^2)$, $J(u)=\Vert u \Vert_1$, and choose $K$ such that
\begin{equation}
	K^*K = \left( \begin{array}{cc} 1 & 2\epsilon \\ 2 \epsilon & \epsilon \end{array} \right)
\end{equation}
for $0 < \epsilon < \frac{1}4$. Then it is straight-forward to see that the ground state is given by 
$u_0=\pm(1,0)^T$ and hence the only choice for $u_1$ orthogonal to $u_0$ is given by $u_1 = \pm (0,1)^T$. Without restriction of generality we can use the one with positive sign and verify that it is not a singular vector. We find
$$ \lambda K^*K u_1 = \lambda (2\epsilon,\epsilon)^T. $$ 
Now $\lambda K^*K u_1 \in \partial \Vert u_1 \Vert_1$ if and only if its second entry equals one and the absolute value of the first entry is less or equal one. We thus find the conditions $\lambda \epsilon = 1$ and $2 \lambda \epsilon \leq 1$, which cannot be satisfied, implying that $u_1$ is not a singular vector. 
\end{exm}

From the last example we see that it is not possible to construct an orthonormal basis of singular vectors starting from the ground state in general. However, in certain cases this may still be possible, as we shall see also in a surprising example in the next section.

%
%

\section{Scales and Scale Estimates}

In the linear singular value decomposition, i.e. $\lambda K^* K u = \frac{u}{\Vert u\Vert_K}$ in our setup, the singular vectors and singular values carry information about scale. The usual definition of scale respectively frequency is related to the value of $\lambda$. In analogy to the eigenvalues of the Laplace operator mentioned above, small $\lambda$ means small frequency respectively large scale, and the scale information is carried by the singular vectors. For increasing $\lambda$ the frequency is increasing, respectively the scale is decreasing. It turns out that the interpretation of scale is even more striking in some nonlinear regularization cases, in particular for the ROF denoising model discussed in Example
\ref{exm:singvec}. Here, intuitively the ground state is the largest scale, since it includes two plateaus of size $\frac{1}2$. For other values of $a$ we have $\lambda = \frac{1}{a(1-a)}$, i.e. decreasing scale the farther $a$ is away from $\frac{1}2$. This is quite intuitive as well, since the singular value has plateaus of length $a$ or $1-a$ and the size of the smaller plateau is decreasing in scale as $\lambda$ is increasing. As we shall see in the following, we can use an appropriate selection of singular vectors of the ROF problem to obtain a standard multiscale representation.

\subsection{Scales in Total Variation and the Haar Wavelet}

The probably most frequently studied and best understood multiscale decomposition is the one of signals in the Haar wavelet basis, given on the unit interval by $\psi_{0,0}(x) = \chi(x)$ and 
\begin{equation}
	\psi_{j,k}(x)=2^{(j-1)/2}(\chi(2^jx-2k)-\chi(2^jx-2k+1)) , \qquad k=0,\ldots,2^{j-1}-1, \ j=1,2,\ldots 
\end{equation}
with the scale function $\chi$ being the characteristic function of the unit interval, i.e.
\begin{equation}
	\chi(x) = \left\{ \begin{array}{ll}1 &\text{if } x\in [0,1] \\ 0 & \text{else} \end{array} \right. \, \text{.}
\end{equation}
It is a kind of folklore that the Haar wavelet decomposition is closely related to total variation methods in spatial dimension one (respectively to anisotropic total variation in higher dimension), and rough connections between ROF denoising and filtering with Haar wavelets have been established (cf. \cite{cai,kamilov,steidl}). Here we shall provide a more explicit connection between the Haar wavelet basis and singular vectors of the ROF functional. For this sake we use a slightly nonstandard definition of the total variation in the form
\begin{align}
\tv_*(u) := \sup_{\substack{\varphi \in C^\infty(\Omega; \R^n) \\ \|
\varphi \|_{L^\infty(\Omega; \R^n)} \leq 1}} \int_\Omega u ~\div \varphi ~dx  \,
\text{.}
\end{align}
Note that in contrast to \eqref{eq:totalvar} we do not choose test functions $\varphi$ with compact support, which yields an additional boundary term. For functions $u \in W^{1,1}([0,1])\cap C([0,1])$ it is straight-forward to see that 
\begin{equation}
	\tv_*(u) = \int_0^1 \left|u^\prime(x)\right|~dx + |u(1)| + |u(0)|.
\end{equation}
This in particular eliminates the kernel of the total variation, so that the ground state is indeed a constant $u_0 \equiv 1$ with $\lambda_0  = \tv_*(u_0)=2$. Then our main result is the following:

\begin{thm}
Let $K:\text{BV}([0,1])\rightarrow L^2([0,1])$ be the embedding operator and let $J=\tv_*$. Then the Haar wavelet basis is an orthonormal set of singular vectors for $K$ and $J$, i.e.
\begin{equation}
	\lambda_{j,k} \psi_{j,k} \in \partial \tv_*(\psi_{j,k}) 
\end{equation}
with singular value $\lambda_{j,k}=2^{(j+3)/2}$ for $j\geq 1$ and $\lambda_0=2$. In particular $u_0$ is a ground state.
\end{thm} 
\begin{proof}
We first show that $u_0=\psi_{0,0}$ is a ground state. Take an arbitrary continuously differentiable function $u$ satisfying the normalization condition $\int_0^1 u^2~dx = 1.$ Then there exists $x_0 \in [0,1]$ such that $|u(x_0)| \geq 1$. Thus, by the triangle inequality
$$ \tv_*(u) \geq |u(1)-u(x_0)|  + |u(x_0)-u(0)| + |u(1)| + |u(0)| \geq 2|u(x_0)| \geq 2 $$
holds. Since $C^1([0,1])$ is dense in BV$([0,1])$ we conclude that the infimum over $\tv_*$ over all functions of bounded variation with normalized $L^2$-norm is $2$. Since $\tv_*(u_0)=2$ and $\Vert u_0 \Vert_{L^2([0, 1])} = 1$, it is a ground state. 

We further need to verify that $p_{j,k}=\lambda_{j,k}\psi_{j,k}$ is a subgradient of the one-homogeneous functional $\tv_*$ at $\psi_{j,k}$, for $j\geq 1$. 
For this sake we first compute $\tv_*(\psi_{j,k})=2^{(j-1)/2} 4 = 2^{(j+3)/2}$ and immediately see that
$$ \langle p_{j,k} , \psi_{j,k} \rangle = \lambda_{j,k} \int_0^1 |\psi_{j,k}|^2~dx = 2^{(j+3)/2} . $$
The remaining step is to prove that 
$$ \langle p_{j,k} , u \rangle \leq \tv_*(u)$$
for arbitrary $u \in BV([0,1])$. For this sake we consider the primitive $q$ satisfying $q^\prime=p_{j,k}$ and
$q(0)=-1$. We observe that indeed $q$ attains its maxima and minima on the jump set of $p_{j,k}$ and they equal $+1$ respectively $-1$. Hence, $\Vert q \Vert_\infty = 1$ and we find
$$ \langle p_{j,k} , u \rangle  = \int_0^1 q^\prime(x)u(x)~dx \leq \sup_{\substack{\varphi \in W^{1,2}([0,1]; \R) \\ \|
\varphi \|_{L^\infty([0,1])} \leq 1}} \int_\Omega u ~\varphi^\prime ~dx. $$
Finally, by a density argument we conclude that this supremum equals the one over $C^\infty$, hence
$$ \langle p_{j,k} , u \rangle \leq  \sup_{\substack{\varphi \in C^\infty([0,1]; \R) \\ \|
\varphi \|_{L^\infty([0,1])} \leq 1}} \int_\Omega u ~\varphi^\prime ~dx = \tv_*(u). $$
\end{proof}

\subsection{Scale Estimates}

In the following we want to demonstrate how singular vectors can be used to derive sharp error estimates on the components at different scales without any additional prior knowledge on the solution, as e.g. the assumption of specific source conditions (cf. \cite{BO04,hofmann}). We will make the connection to scale explicit by again considering the case of the ROF-model and the corresponding singular vectors.

Let us assume we are given a singular vector $u_\lambda$ with singular value
$\lambda$ and dual singular vector $p_\lambda$, i.e. 
$$\lambda K^{*}Ku_\lambda = p_\lambda \in \partial J(u_\lambda)$$ 
for $\|Ku\|_{\hmc} = 1$.
Then we can estimate solutions of \eqref{eq:varframe} for input data given in
terms of $f = K\tilde{u} + \eta$, with $\tilde{u} \in \dom(K)\cap\dom(J)$ and $\eta \in
\hmc$, with respect to this particular singular vector.
\begin{thm}\label{thm:errest}
For input data $f = K\tilde{u} + \eta$, $\tilde{u} \in \dom(K)$
and $\eta \in \hmc$, the solution $\hat{u}$ of \eqref{eq:varframe} satisfies the
estimate
\begin{align*}
\left| \left\langle \frac{1}{\lambda} p_\lambda, \hat{u} - \tilde{u}
\right\rangle_{\umc} \right| \leq \left| \left\langle \eta,
Ku_\lambda\right\rangle_{\hmc} \right| + \alpha \left| \left\langle \hat{p},
u_\lambda \right\rangle_{\umc} \right|  \, \text{.}
\end{align*}
\begin{proof}
The estimate is simply a consequence of taking a dual product of the optimality condition of \eqref{eq:varframe} with $u_\lambda$. The optimality condition reads as
\begin{align*}
K^{*}K(\hat{u} - \tilde{u}) + \alpha \hat{p} &= K^{*}\eta \, \text{,}
\end{align*}
for $\hat{p} \in \partial J(\hat{u})$. Taking the dual product with $u_\lambda$ yields
\begin{align*}
\langle K^{*}Ku_\lambda , \hat{u} - \tilde{u} \rangle_{\umc} = \langle
K^{*}\eta - \alpha \hat{p}, u_\lambda \rangle_{\umc} \, \text{,}
\end{align*}
which is equivalent to
\begin{align*}
\left\langle \frac{1}{\lambda} p_\lambda , \hat{u} - \tilde{u}
\right\rangle_{\umc} = \langle K^{*}\eta - \alpha \hat{p}, u_\lambda
\rangle_{\umc} \, \text{.}
\end{align*}
The estimate for the absolute value of the right-hand side simply follows from the triangular inequality.
\end{proof}
\end{thm}
\begin{rem}\normalfont
In case of $J$ being one-homogeneous we have $\left| \left\langle \hat{p},
u_\lambda \right\rangle_{\umc} \right| \leq J(u_\lambda) = \lambda$ and the
estimate therefore can be used to deduce
\begin{align*}
\left| \left\langle \frac{1}{\lambda} p_\lambda, \hat{u} - \tilde{u}
\right\rangle_{\umc} \right| \leq \alpha \lambda + \left| \left\langle \eta,
Ku_\lambda\right\rangle_{\hmc} \right| \, \text{.}
\end{align*}
The two terms in the error estimate can be easily interpreted. The term $\alpha \lambda$ is a worst-case estimate for the bias at the scale of the used singular value, while the second term describes the impact of noise on a specific scale. Note that the scalar product of the noise with $Ku_\lambda$ implies actually some averaging of the noise, which usually reduces the noise variance stronger on larger scales than on smaller ones. This can be analyzed in particular for statistical noise models such as additive Gaussian noise.  
\end{rem}
In the following we want to demonstrate how Theorem \ref{thm:errest} can be used to derive estimates for the difference of the reconstruction $\hat{u}$ and the input data $\tilde{u}$ on different scales.
\begin{exm}\normalfont
With this first example we want to find a worst-case estimate for a
reconstruction $\hat{u}$ of \eqref{eq:rof}, with respect to the input data $f = \tilde{u} + \eta$ on the scale
$[0, a] \subset [0, 1]$, for $0 < a < 1$. We want to point out that the ROF
model is mean-value preserving, i.e. $\int_0^1
(\hat{u} - \tilde{u}) ~dx = 0$. Let us investigate the integral equation
\begin{align}
\int_0^a (\hat{u} - \tilde{u}) ~dx &= c_0 \int_0^1 u_0 \left( \hat{u} - \tilde{u} \right) ~dx + c_1 \int_0^1 u_1 \left( \hat{u} - \tilde{u} \right) ~dx\,\text{,}\nonumber\\
 &= \int_0^1 \left(c_0  u_0 + c_1 u_1\right)\left( \hat{u} - \tilde{u} \right) ~dx\,\text{,}\nonumber\\
 &= \int_0^a \left(c_0  u_0 + c_1 u_1\right)\left( \hat{u} - \tilde{u} \right) ~dx + \int_a^1 \left(c_0  u_0 + c_1 u_1\right)\left( \hat{u} - \tilde{u} \right) ~dx \, \text{,}\label{eq:intrel}
\end{align}
with the function $u_0$ defined as $u_0(x) \equiv 1$, and with $u_1(x) := u^a(x)$ being the singular vector of Example \ref{exm:singvec}. Then, \eqref{eq:intrel} reads as
\begin{align}
\int_0^a (\hat{u} - \tilde{u}) ~dx = \left(c_0 - \sqrt{\frac{1-a}{a}}c_1\right) \int_0^a (\hat{u} - \tilde{u}) ~dx + \left(c_0 + \sqrt{\frac{a}{1-a}}c_1\right) \int_a^1 ( \hat{u} - \tilde{u} ) ~dx\, \text{.}\label{eq:intrellse}
\end{align}
In order to satisfy \eqref{eq:intrellse}, the coefficients $c_0$ and $c_1$ have to be chosen such that they solve the linear system of equations
\begin{equation}
\left(
\begin{array}{cc}
1 & - \sqrt{\frac{1-a}{a}}\\
1 & \sqrt{\frac{a}{1-a}}
\end{array}
\right) \left( \begin{array}{c}
c_0\\
c_1
\end{array} \right) = \left( \begin{array}{c}
1\\
0
\end{array} \right) \, \text{.} \label{eq:lse}
\end{equation}
It is easy to see that $c_0 = a$ and $c_1 = - \sqrt{a(1 - a)}$ solve \eqref{eq:lse}. Thus, we obtain the following scale estimate
\begin{align*}
\left| \int_0^a (\hat{u} - \tilde{u}) ~dx \right| &= \left|a \underbrace{\int_0^1 \hat{u} - \tilde{u}~dx}_{= 0} -\sqrt{a(1 - a)} \int_0^1 u_1 \left(\hat{u} - \tilde{u}\right)~dx \right| \, \text{,}\\
&= \sqrt{a(1 - a)} \left| \int_0^1 u_1 \left(\hat{u} - \tilde{u}\right)~dx \right| \, \text{,}\\
&\leq \sqrt{a(1 - a)} \left( \frac{\alpha}{\sqrt{a(1 - a)}} + \left| \int_0^1 \eta ~u_1 ~dx \right| \right) \, \text{,}\\
&= \left(\alpha + \sqrt{a(1 - a)} \left| \int_0^1 \eta ~u_1 ~dx \right| \right)
\, \text{,}
\end{align*}
since $1/\sqrt{a(1 - a)}$ is the singular value of $u_1$. Note that for $\eta = 0$ we simply obtain
\begin{align*}
\left| \int_0^a ( \hat{u} - \tilde{u} ) ~dx \right| \leq \alpha \, \text{.}
\end{align*}
\begin{figure}[!ht]
\begin{center}
\includegraphics[width=0.6\textwidth]{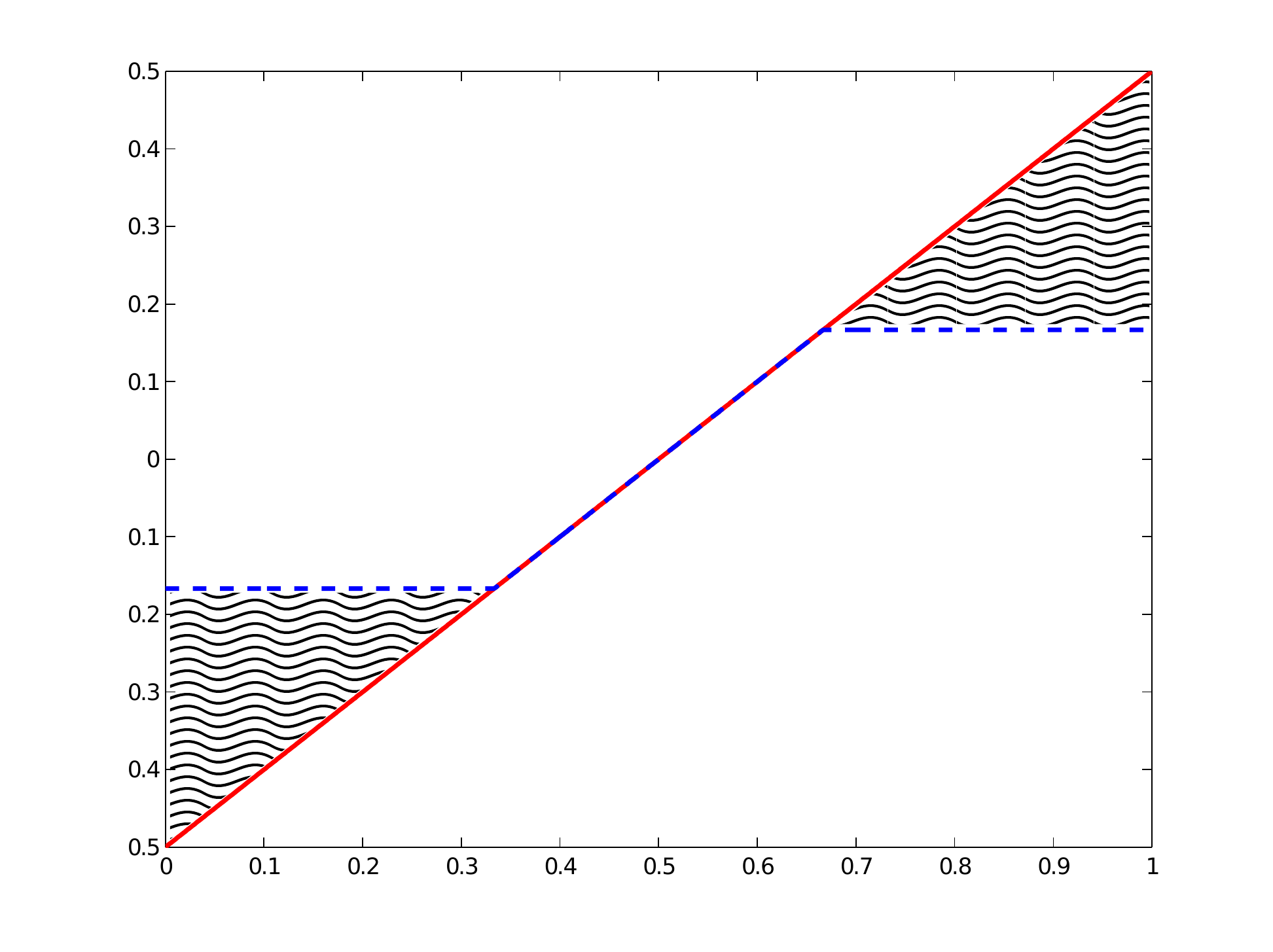}
\end{center}
\caption{The function $\tilde{u}(x) = x - 1/2$ (solid red line) and the solution of \eqref{eq:rof} for input data given in terms of $f = \tilde{u}$ and $\alpha = 1/18$ (dashed blue line). The areas indicated by wavy lines equal $-1/18$ and $1/18$, respectively.}
\label{fig:scalest}
\end{figure}
This estimate is obviously not sharp if $a \rightarrow 1$, since $\int_0^1
\hat{u} - \tilde{u} ~dx = 0$, but gives a worst case estimate if $a$ is far away
from the boundary as in case of $a = 1/2$. If we consider e.g. $\tilde{u}(x) = x
- 1/2$, the estimate guarantees that the area  visualized in Figure \ref{fig:scalest}
is equal or smaller than $\alpha$. Indeed the area equals $\alpha$, which
becomes clear by computing the exact solution of \eqref{eq:rof} for $f(x) =
\tilde{u}(x) = x - 1/2$. For $0 < \alpha < 1/8$ the solution
satisfies
\begin{align*}
\hat{u}(x) = \begin{cases} \sqrt{2\alpha} - \frac{1}{2} & x \in \left[0,
\sqrt{2\alpha}\right[\\ x - \frac{1}{2} & x \in \left[\sqrt{2\alpha}, 1 - \sqrt{2\alpha}\right]\\
\frac{1}{2} - \sqrt{2\alpha} & x \in \left]1 - \sqrt{2\alpha}, 1\right]
\end{cases} \, \text{,}
\end{align*}
and thus, $\int_0^{1/2} (\hat{u} - \tilde{u}) ~dx = - \int_{1/2}^1 (\hat{u} -
\tilde{u} )~dx = \alpha$ holds true.
\end{exm}
\begin{exm}\normalfont
As a second example we want to focus on a scale estimate on the scale $[a, b]
\subset [0, 1]$, for $0 < a < b < 1$. Again we study \eqref{eq:rof} and, in analogy to the previous example, consider
\begin{align}
\int_a^b \hat{u} - \tilde{u} ~dx = \int_0^1 \left( c_0 u_0 + c_1 u_1 + c_2 u_2 \right) \left(\hat{u} - \tilde{u}\right) ~dx\, \text{,}\label{eq:intrel2}
\end{align}
with $u_0(x) \equiv 1$, $u_1(x) := u^a(x)$ and $u_2(x) := -u^b(x)$, for $u^a$ and $u^b$ being singular vectors as defined in Example \ref{exm:singvec}. Thus, we can rewrite \eqref{eq:intrel2} to
\begin{align*}
\int_a^b (\hat{u} - \tilde{u}) ~dx {} = {} &\int_0^a \left( c_0 u_0 + c_1 u_1 +
c_2 u_2 \right) \left(\hat{u} - \tilde{u}\right) ~dx ~+\\
&\int_a^b \left( c_0 u_0 + c_1 u_1 + c_2 u_2 \right) \left(\hat{u} -
\tilde{u}\right) ~dx ~+\\
&\int_b^1 \left( c_0 u_0 + c_1 u_1 + c_2 u_2 \right) \left(\hat{u} - \tilde{u}\right) ~dx \\
{} = {} &\left( c_0 - \sqrt{\frac{1-a}{a}}c_1 + \sqrt{\frac{1-b}{b}}c_2\right)
\int_0^a (\hat{u} - \tilde{u} )~dx ~+\\
& \left( c_0 + \sqrt{\frac{a}{1-a}}c_1 + \sqrt{\frac{1-b}{b}}c_2\right) \int_a^b
(\hat{u} - \tilde{u})~dx ~+\\
& \left( c_0 + \sqrt{\frac{a}{1-a}}c_1 - \sqrt{\frac{b}{1-b}}c_2\right) \int_b^1 (\hat{u} - \tilde{u}) ~dx \, \text{.}
\end{align*}
Similar to the previous example we therefore have to make sure that $c_0$, $c_1$ and $c_2$ satisfy
\begin{equation}
\left( \begin{array}{ccc}
1 & - \sqrt{\frac{1-a}{a}} & \sqrt{\frac{1-b}{b}}\\
1 & \sqrt{\frac{a}{1-a}} & \sqrt{\frac{1-b}{b}}\\
1 & \sqrt{\frac{a}{1-a}} & - \sqrt{\frac{b}{1-b}}
\end{array} \right) \left( \begin{array}{c}
c_0\\
c_1\\
c_2
\end{array} \right) = \left( \begin{array}{c}
0\\
1\\
0
\end{array} \right) \, \text{.}
\label{eq:lse2}
\end{equation}
Solving \eqref{eq:lse2} for $c_0$, $c_1$ and $c_2$ yields $c_0 = b - a$, $c_1 = \sqrt{a\left(1-a\right)}$ and $c_2 = \sqrt{b\left(1-b\right)}$. Consequently, Theorem \ref{thm:errest} allows us to compute the estimate
\begin{align*}
\left| \int_a^b ( \hat{u} - \tilde{u} )~dx \right| {} \leq {} &(b - a) \left| \underbrace{\int_0^1 \hat{u} - \tilde{u} ~dx}_{= 0} \right| + \sqrt{a\left(1-a\right)} \left| \int_0^1 u_1 \left(\hat{u} - \tilde{u}\right) ~dx \right|\\
 &+ \sqrt{b\left(1-b\right)} \left| \int_0^1 u_2 \left(\hat{u} - \tilde{u}\right) ~dx \right|\\
 {} \leq {} &2\alpha + \sqrt{a\left(1-a\right)} \left| \int_0^1 \eta ~u_1 ~dx
 \right| + \sqrt{b\left(1-b\right)} \left| \int_0^1 \eta ~u_2 ~dx \right| \, \text{.}
\end{align*}
In case of clean data, i.e. $\eta = 0$, we see that
\begin{align*}
\left| \int_a^b (\hat{u} - \tilde{u}) ~dx \right| \leq 2\alpha
\end{align*}
holds. It is remarkable that the worst-case estimate on the scale $[a, b]$ does not depend on the boundary values $a$ and $b$. However, if we want to estimate the mean value, then an additional factor $\frac{1}{b-a}$ comes up, which increases with decreasing size of the interval. 
\end{exm}

\section{Exact Reconstruction of Singular Vectors}\label{sec:exactrec}

In this section we discuss exact solutions of variational regularization methods, namely the recovery of singular vectors. We will relate this results to properties of Bregman distances. First of all, we want to recall a criterion that has been shown by Meyer in \cite{meyer} in order to derive trivial ground states for the ROF-model. These considerations can be generalized by the use of Bregman distances as it can be seen by the following theorem.
\begin{thm}\label{thm:meyercond}
Let $f$ be such that
\begin{align}
\frac{1}{\alpha} K^{*}f \in \partial J(0)\label{eq:zerocond}
\end{align}
is satisfied, then a minimizer of \eqref{eq:varframe} is given by $\hat{u} \equiv 0$. Vice versa, $\hat u \equiv 0$ is a minimizer of  \eqref{eq:varframe} only if \eqref{eq:zerocond} holds.
\begin{proof}
We can rewrite \eqref{eq:varframe} to
\begin{align*}
\hat{u} \in  \argmin_{u \in \dom(J)} \left\{\ltwon[\hmc]{1}{Ku} + \alpha \left(
J(u) - \dupr[\umc]{\frac{1}{\alpha}K^{*}f}{u} \right) +
\ltwon[\hmc]{1}{f}\right\}
\end{align*}
Since \eqref{eq:zerocond} is satisfied, we can define $q := (K^{*}f)/\alpha$ such that
\begin{align*}
D_J^q(u, 0) = J(u) - J(0) - \dupr[\umc]{q}{u}
\end{align*}
is a non-negative Bregman distance. Hence, ignoring the constant part $1/2
\|f\|_{\hmc}$ we have
\begin{align*}
\hat{u} \in  \argmin_{u \in \dom(J)} \left\{\ltwon[\hmc]{1}{Ku} + \alpha D_J^q(u,
0) \right\}
\end{align*}
for which the obvious minimizer is given via $\hat{u} = 0$, since both terms are non-negative and vanish for $u = 0$. It is straightforward to see the opposite condition from the optimality condition for $\hat{u}=0$.
\end{proof}
\end{thm}
\begin{rem}\normalfont
Note that if \eqref{eq:zerocond} is satisfied for a specific $\tilde{\alpha}$, then \eqref{eq:zerocond} is automatically
guaranteed for every $\alpha \geq \tilde{\alpha}$, since $(K^{*}f)/\tilde{\alpha} \in \partial J(0)$ implies
\begin{align*}
J(v) &\geq \dupr[\umc]{\frac{1}{\tilde{\alpha}} K^{*}f}{v} \, \text{,}
\intertext{for all $v \in \dom(J)$. If we multiply both sides of the inequality with $\tilde{\alpha}$ we obtain}
\tilde{\alpha} J(v) &\geq \dupr[\umc]{K^{*}f}{v} \, \text{,}
\intertext{since $\tilde{\alpha}$ is positive. Due to the positivity of $J$ we even have}
\alpha J(v) &\geq \tilde{\alpha} J(v)
\end{align*}
for all $v \in \dom(J)$ and $\alpha \geq \tilde{\alpha}$, and hence, \eqref{eq:zerocond} is guaranteed for all $\alpha \geq \tilde{\alpha}$.
\end{rem}
Theorem \ref{thm:meyercond} yields an explicit condition on the regularization
parameter $\alpha$ to enforce the solution of \eqref{eq:varframe} to be zero.
Furthermore, according to the following Lemma for singular vectors $\ul$
of one-homogeneous functionals there even have to exist parameters $\alpha$ such
that \eqref{eq:zerocond} is fulfilled for $f = K\ul$.

\begin{lem}\label{lem:meyercondeigfct}
Let $J$ be one-homogeneous. If $u \in \dom(J) \cap \dom(K)$ is
a function such that \eqref{eq:zerocond} does not hold for any $\alpha \in \rpos$ with data $f = Ku$, i.e.
\begin{align*}
\frac{1}{\alpha} K^{*}Ku \notin \partial J(0) \ ~\forall \alpha \in \rpos \, \text{,}
\end{align*}
then, $u$ is not a singular vector with singular value $\lambda \neq 0$.
\begin{proof}
We want to prove the statement by contradiction. We therefore assume that on the one hand, $u$ is a singular vector with singular value $\lambda$, i.e. $\lambda K^{*}Ku = p \in \partial J(u)$. Taking a duality product of this relation with $u$ yields the equality
\begin{align}
\lambda \| Ku \|^2_{\hmc} = J(u) \, \text{,}\label{eq:meylemproof1}
\end{align}
due to the one-homogeneity of $J$. Moreover, from the definition of the
subdifferential, the singular value property yields
\begin{align}
J(v) \geq J(u) + \lambda \dupr[\umc]{K^{*}Ku}{v - u} \ \, \forall v \in \dom(J)
\,
\text{.}\label{eq:meylemproof2}
\end{align}
On the other hand, we know due to $\left(K^{*}Ku\right)/\alpha \notin J(0)$ for all $\alpha \in \rpos$ that there has to exist a function $v \in \dom(J)$ with
\begin{align}
\dupr[\hmc]{Ku}{Kv} &> \alpha J(v) \, \text{.}\label{eq:meylemproof3}
\intertext{If we insert \eqref{eq:meylemproof2} into \eqref{eq:meylemproof3}, for the particular choice of $v$ we therefore obtain}
\dupr[\hmc]{Ku}{Kv} &> \alpha \left( J(u) + \lambda \dupr[\hmc]{Ku}{Kv} -
\lambda \| Ku \|^2_{\hmc} \right)\nonumber\\
\Leftrightarrow (1 - \lambda\alpha)\dupr[\hmc]{Ku}{Kv} &> \alpha \left( J(u) -
\lambda \| Ku \|^2_{\hmc} \right) \, \text{.}\label{eq:meylemproof4}
\intertext{Equation \eqref{eq:meylemproof4} is supposed to be true for every $\alpha \in \rpos$, especially for the particular choice $\alpha = 1/\lambda$. In this case, \eqref{eq:meylemproof4} reads as}
\lambda \| Ku \|^2_{\hmc} &> J(u) \, \text{,}\nonumber
\end{align}
for $\lambda > 0$, and therefore is a contradiction to \eqref{eq:meylemproof1}.
\end{proof}
\end{lem}

The equivalent reverse statement of Lemma \ref{lem:meyercondeigfct} is that for every data $f$ created by a singular vector $u_\lambda$, i.e. $f = Ku_\lambda$, there exists a parameter $\tilde{\alpha}$ such that \eqref{eq:zerocond} is valid for $\alpha \geq \tilde{\alpha}$.
Moreover, condition \eqref{eq:zerocond} guarantees that the data $f$ needs to satisfy certain properties in order to vanish for a large regularization parameter $\alpha$, e.g. $f$ does need to have zero mean for $K = I$ in the case of $\tv$-regularization.

\begin{rem}\normalfont
Note that, however, it is possible for a particular function $f$ that there does
not exist a parameter $\alpha$ such that \eqref{eq:zerocond} is fulfilled,
although $f = Ku_0$ is given in terms of a trivial ground state (and
therefore in terms of a singular vector) with singular value $\lambda = 0$. A
simple example would be that $f \not\equiv 0$ is a constant function in case
of $K = I$ and $J = \tv$.
\end{rem}


\subsection{Clean Data}\label{subsec:exrecofeigfct}
In case of clean data $f = \gamma K\ul$, $\gamma > 0$ and $\ul$ being a
non-trivial singular vector, we are interested in finding a solution of
\eqref{eq:varframe} that can be expressed in terms of this singular vector, i.e.
$\hat{u} = c \tilde{u}$ for a positive constant $c$. We want to call such a
function \textit{almost exact solution}. 

\noindent The following theorem gives us the conditions on $\alpha$ needed for recovering
a multiple of $\ul$.
\begin{thm}\label{thm:eigfctclean}
Let $J$ be one-homogeneous. Furthermore, let $\ul$ be
a singular vector with corresponding singular value $\lambda$. Then, if the data
$f$ is given as $f = \gamma K\ul$ for a positive constant $\gamma$, a 
solution of \eqref{eq:varframe} is $\hat{u} = c \ul$ for
\begin{align*}
	c = \gamma - \alpha \lambda \, \text{,}
\end{align*}
if $\gamma > \alpha \lambda$ is satisfied. \end{thm}
\begin{proof}
Again, we rewrite \eqref{eq:varframe} in terms of a Bregman distance. Inserting
$f = \gamma K\ul$ yields
\begin{align*}	
\hat{u} {} \in {} &\argmin_{u \in \dom(J)} \left\{ \ltwon[\hmc]{1}{Ku - \gamma
K\ul} + \alpha J(u) \right\}\\
{} \in {} &\argmin_{u \in \dom(J)} \left\{ \ltwon[\hmc]{1}{Ku - c K\ul} +
\alpha J(u) + \alpha J(c \ul) - \frac{\gamma - c}{\lambda}
\dupr[\umc]{\lambda K^{*}K\ul}{u} \right.\\
&\left. \hphantom{\argmin_{u \in \dom(J)} \{} + \frac{1}{2} \left(
\dupr[\hmc]{\gamma K\ul}{\gamma K\ul} + \dupr[\hmc]{c
K\ul}{c K\ul}  \right) - \alpha J(c \ul) \right\} \, \text{.}
\end{align*}
By ignoring the constant part, for $\gamma > \alpha \lambda$ and $c = \gamma - \alpha \lambda > 0$ we therefore obtain
\begin{align*}
\hat{u} {} = {} &\argmin_{u \in \dom(J)} \left\{ \ltwon[\hmc]{1}{Ku - c
K\ul} + \alpha D_J^q(u, c\ul) \right\} \, \text{,}\\
\intertext{with}
q {} = {} &\lambda K^{*}K\ul \in \partial J(\ul) \stackrel{\text{$J$ one-homogeneous}}{=} \partial J(c \ul) \, \text{.}
\end{align*}
An obvious minimizer is $\hat{u} = c \ul$.
\end{proof}

Note that the above result does not yield that the singular value is the unique minimizer, except for $K$ having trivial nullspace.
To see this, let us consider model \eqref{eq:varframe} with $J(u) = \| u \|_{\ell^1}$ and $K$ being
the matrix
\begin{align*}
K := \frac{1}{\sqrt{2}} \left( \begin{array}{cc} 1 & 1\\ 1 & 1
\end{array}\right) \, \text{.}
\end{align*}
It is obvious that $K$ is normalized with respect to the $\ell^2$-norm, but
neither is injective nor surjective. Due to Example \ref{exm:ell1gs} both $e_1$
and $e_2$ are singular vectors. However, both yield the same output $f =
(1/\sqrt{2}, 1/\sqrt{2})^T$ and therefore both $\hat{u} =
(1 - \alpha) e_1$ and $\hat{u} = (1 - \alpha) e_2$ satisfy Theorem
\ref{thm:eigfctclean}.

We also mention that the main line of Theorem \ref{thm:eigfctclean} also holds for $p$-homogeneous functionals, but with different constants $c$. In the following we turn our attention to noisy data, where the one-homogeneity is much more essential, e.g. exact reconstruction for a wide class of noise realizations cannot hold for quadratic regularizations.

\subsection{Noisy Data}\label{subsec:exrecofeigfctnoisy}

The multivaluedness of the subdifferential $\partial J$ allows to obtain almost
exact solutions even in the presence of noisy data, i.e. $f = \gamma K\ul + n$,
though the case of noisy data is slightly more complicated to prove. If the most
significant features of $\ul$ with respect to the regularization energy $J$ are
left unaffected by the noise, then the following theorem guarantees almost exact
recovery of the singular vector $\ul$.
\begin{thm}\label{thm:eigfctnoisy}
Let $J$ be one-homogeneous. Furthermore, let $\ul$ be
a singular vector with corresponding singular value $\lambda$. The data $f$ is
assumed to be corrupted by noise $n$, i.e. $f = \gamma K\ul + n$ for a positive constant $\gamma$, such that there exist positive constants $\mu$ and $\eta$ with
\begin{align}
\mu K^{*}K\ul + \eta K^{*}n \in \partial J(\ul) \, \text{.}\label{eq:noisyeigcond}
\end{align} 
Then, a solution of \eqref{eq:varframe} is given by $\hat{u} = c \ul$ for
$$
	c = \gamma - \alpha \lambda + \frac{\lambda - \mu}{\eta} \, \text{,} $$
if $\gamma$ satisfies the SNR-condition
\begin{equation}
	\gamma > \frac{\mu}{\eta} \, \text{,}		\label{eq:snrcond}
\end{equation}
and if $\alpha \in [1/\eta, \gamma/\lambda + 1/\eta [$ holds.
\begin{proof}
Similar to the proof of Theorem \ref{thm:eigfctclean} we rewrite \eqref{eq:varframe} to
\begin{align*}
\hat{u} &= \argmin_{u \in \dom(J)} \left\{ \ltwon[\hmc]{1}{Ku - \gamma K\ul - n} + \alpha J(u) \right\}\\
 &= \argmin_{u \in \dom(J)} \left\{ \ltwon[\hmc]{1}{Ku - c K\ul} + \alpha J(u) - \alpha \dupr[\umc]{\frac{\gamma - c}{\alpha}K^{*}K\ul + \frac{1}{\alpha} K^{*}n}{u} \right\} \\
&= \argmin_{u \in \dom(J)} \left\{ \ltwon[\hmc]{1}{Ku - c K\ul} + \alpha D_J^q(u, c\ul) \right\} \, \text{,}
\end{align*}
with obvious minimizer $\hat{u} = c \ul$, if we neglect the constant parts and if we can manage to choose $c$ such that
\begin{align*}
\frac{\gamma - c}{\alpha}K^{*}K\ul + \frac{1}{\alpha} K^{*}n \in \partial J(\ul) = \partial J(c \ul) \, \text{.}
\end{align*}
Note that since $\partial J(\ul)$ is a convex set not only $\lambda K^{*}K\ul$ and \eqref{eq:noisyeigcond} are elements of $\partial J(\ul)$, but also any convex combination, i.e.
\begin{align*}
\left(\left(1 - \beta\right)\lambda + \beta \mu \right)K^{*}K\ul + \beta \eta K^{*}n \in \partial J(\ul) \, \text{,}
\end{align*}
for each $\beta \in [0, 1]$.

Hence, we need to choose $c > 0$ and $\beta \in [0, 1]$ such that $1/\alpha = \beta \eta$ and $(\gamma - c)/\alpha = (1 - \beta)\lambda + \beta \mu$. Therefore, solutions for $\beta$ and $c$ are
\begin{align*}
\beta &= \frac{1}{\alpha \eta}
\intertext{and}
c &= \gamma - \alpha \lambda + \frac{\lambda - \mu}{\eta} \, \text{.}
\end{align*}
In order to satisfy $\beta \leq 1$ and $c > 0$, $\alpha$ has to be chosen such that $\alpha$ is bounded via
\begin{align*}
\frac{1}{\eta} \leq \alpha < \frac{\gamma}{\lambda} + \frac{1}{\eta} + \frac{\mu}{\lambda \eta} \, \text{.}
\end{align*}
This condition can only be satisfied, if $\gamma > \mu/\eta$ holds.
\end{proof}
\end{thm}

At a first glance \eqref{eq:noisyeigcond} seems unreasonably restrictive, e.g. for smooth functionals $J$ it can only hold if the noise is generated by the singular vector itself. This however differs completely in the situation of singular regularization functionals with large subdifferentials. To see this, let us consider the setup of Example \ref{exm:ell1gs} with a ground state $e_i$, together with the reasonable assumptions that no pair of columns in $K$ is linearly independent. For 
$$ p = \lambda K^*K e_i \in \partial \Vert e_i \Vert_{\ell^1(\R^N)} $$
we have 
$$p_i = p \cdot e_i =\lambda e_i \cdot K^*K e_i = \lambda = 1$$
and hence for $j \neq i$ 
$$\vert p \cdot e_j \vert =  \vert e_j \cdot K^*K e_i \vert = \vert K_j \cdot K_i \vert < 1, $$
where we denote by $K_j$ the $j$-th column of $K$. Now let $n \in \R^M$ be the noise vector and $v=K^* n$. 
In order to satisfy  \eqref{eq:noisyeigcond} we need 
$$ 1= p \cdot e_i = \mu e_i \cdot K^*K e_i  + \eta e_i \cdot v = \mu + \eta v_i, $$
and 
$$ 1 \geq \vert p \cdot e_j \vert = \vert \mu K_j \cdot K_i + \eta v_j \vert. $$ 
The first condition is satisfied if $\mu = 1 - \eta v_i$. Since then
$$ \vert \mu K_j \cdot K_i + \eta v_j \vert =  \vert K_j \cdot K_i \vert + {\cal O}(\eta) 
$$  
for $\eta$ small and $ \vert K_j \cdot K_i \vert < 1$, we indeed conclude that there always exists $\eta$ such that \eqref{eq:noisyeigcond} is satisfied. Note that in this case $\lambda = 1$ and hence $c= \gamma - \alpha + v_i$.

\section{Bias of Variational Methods}

We have seen in the previous section that there remains a small bias in the exact reconstruction of singular vectors, which is incorporated in the fact that $c < \gamma$, with difference depending on the actual singular value. For $f = \gamma Ku_\lambda$, this difference yields a residual
$$ \Vert Ku - f \Vert_{\hmc} = \alpha \lambda \Vert Ku_\lambda \Vert_{\hmc} = \alpha \lambda, $$
i.e. a bias in the solution, which is minimal for the ground state $\lambda_0$. In this short section we will show that indeed this bias prevails for arbitrary data and the residual is bounded below by $\alpha \lambda_0$, which again confirms the extremal role of the ground state. 

\begin{thm}
Let $J$ be one-homogeneous, $f \in \hmc$ be arbitrary, $\alpha > 0$, and let 
$$u^\alpha \in \argmin_{u \in \umc} \left( \frac{1}2 \Vert Ku - f \Vert_{\hmc}^2 + \alpha J(u) \right). $$ 
Then 
\begin{equation}
	\Vert Ku^\alpha \Vert_{\hmc} \leq \max\{\Vert f \Vert_{\hmc} - \alpha \lambda_0, 0\},
\end{equation}
where $\lambda_0$ is the smallest singular value. As a direct consequence, if $\Vert f \Vert_{\hmc} \geq \alpha \lambda_0$, then
\begin{equation}
	\Vert Ku^\alpha -f \Vert_{\hmc} \geq \alpha \lambda_0, \label{residualbias}
\end{equation}
which is sharp if $u^\alpha$ is a multiple of the ground state $u_0$. 
\end{thm}
\begin{proof}
If $Ku^\alpha = 0$, then the estimate is obviously satisfied. Thus, we restrict our attention to the case 
$Ku^\alpha \neq 0$ and define $v:=\frac{u^\alpha}{\Vert Ku^\alpha\Vert}$.
From the dual product of the optimality condition with $v$ we see that for a subgradient 
$p^\alpha \in \partial J(u^\alpha)=\partial J(v)$ that 
$$ \langle K^* K u^\alpha , v \rangle_{\umc} + \alpha \langle p^\alpha, v \rangle_{\umc} = \langle K^* f, v \rangle_{\umc} .$$ 
Due to the one-homogeneity of $J$ we conclude $\langle p^\alpha, v \rangle_{\umc} = J(v)$ and hence, 
$$ \Vert Ku^\alpha \Vert_{\hmc} + \alpha J(v) = \langle f, Kv \rangle_{\hmc}. $$
By the definition of the ground state we conclude $J(v) \geq \lambda_0$ and by further estimating the right-hand side via the Cauchy-Schwartz inequality $\langle f, Ku_\alpha\rangle_{\hmc} \leq \|f\|_{\hmc} \|Ku_\alpha\|_{\hmc}$ we have
$$  \Vert Ku^\alpha \Vert_{\hmc} + \alpha \lambda_0 \leq \Vert f \Vert_{\hmc}, $$
which implies the assertion.
\end{proof}

The bias in the residual can to some extent also be translated to the regularization functional, as the following result shows:

\begin{thm}
Let $J$ be one-homogeneous, $f = K \tilde u$ for $\tilde u \in \umc$ with $J(\tilde u) < \infty$, such that 
$\Vert f \Vert_{\hmc} \geq \alpha \lambda_0$. Moreover let $\alpha > 0$ and  
$$u^\alpha \in \argmin_{u \in \umc} \left( \frac{1}2 \Vert Ku - f \Vert^2_{\hmc} + \alpha J(u) \right). $$ 
Then 
\begin{equation}
	J(u^\alpha) \leq J(\tilde u) - \frac{\alpha}2 \lambda_0^2 \, \text{,}
\end{equation}
with $\lambda_0$ denoting the smallest singular value.
\end{thm}
\begin{proof}
By the definition of $u^\alpha$ as a minimizer and \eqref{residualbias} we conclude
$$ \frac{1}2 \alpha^2 \lambda_0^2 + \alpha J(u^\alpha) \leq \frac{1}2 \Vert Ku^\alpha -f \Vert^2_{\hmc} 
+ \alpha J(u^\alpha) \leq \alpha J(\tilde u), $$
which yields the assertion after dividing by $\alpha$.
\end{proof}

\section{Unbiased Recovery and Inverse Scale Space}

In Section \ref{sec:exactrec} we have seen that in standard variational methods of the form \eqref{eq:varframe} one can recover singular vector (almost) exactly with a loss of contrast. In this chapter we want to extend this topic to
the question of exact recovery without a loss of contrast, in the absence and
presence of noise. For this sake we are going to investigate the concept of the
inverse scale space flows, which have displayed superior properties to solutions of \eqref{eq:varframe} in several numerical tests (cf. \cite{gilboa,groetsch}).

\noindent Inverse scale space methods can be derived asymptotically from the Bregman iteration 
\begin{equation}
	u^{k+1} \in \argmin_{u \in \dom(J)} \left\{ \frac{1}{2} \left\| Ku - f
\right\|^2_{\hmc} + \alpha ( J(u) - \langle p^k , u \rangle_{\umc})  \right\} \, \text{,}
\end{equation}
for which the subgradient $p^k \in \partial J(u^k)$ satisfies $p^0 \equiv 0$ and 
\begin{equation}
	p^k = p^{k-1} + \frac{1}\alpha K^* (f - Ku^k) \, \text{.}
\end{equation}
In the limit $\alpha \rightarrow \infty$ one can interpret $\Delta t = \frac{1}\alpha$ as a time step tending to zero. Thus, we obtain the inverse scale space flow
\begin{equation}
	\partial_t p(t)= K^* (f - Ku(t)), \qquad p(t) \in \partial J(u(t)). \label{eq:invscale}
\end{equation}
We refer to \cite{gilboa,benning,BFOS07,burgermoeller,aiss,burgresm,moellerdiss} for detailed discussions of the inverse scale space method and its analysis.  

We want to mention that analogous results on exact respectively unbiased reconstruction can be obtained for the Bregman iteration, clearly with some dependence on the value of $\alpha$, further details can be found in \cite{benning}.


\subsection{Clean Data}
Similar to Section \ref{subsec:exrecofeigfct}, we are going to consider data $f = \gamma K\ul$, with $\ul$ being a singular vector. For this
setup we are able to derive the following result:
\begin{thm}\label{thm:invscalspeigclean}
Let $J$ be one-homogeneous and let $\ul$ be
a singular vector with corresponding singular value $\lambda$. Then, if the data
$f$ are given by $f = \gamma K\ul$ for a positive constant $\gamma$, a
solution of the inverse scale space flow \eqref{eq:invscale} is given by
\begin{equation}
	u(t) = \left\{ \begin{array}{ll} 0 & \text{if } t < t_* \\ \gamma \ul & \text{if } t \geq t_* \end{array} \right.
\end{equation}
for $t_{*} = \lambda/\gamma$.
\begin{proof}
First of all we see with Theorem \ref{thm:meyercond} that for $t < t_*$ we have
$$ p(t) = t K^*f = t \gamma K^*K \ul = t \frac{\gamma}\lambda p_\lambda \in \partial J(0). $$
Since $\partial_t p = K^* f$ and $p(0)=0$,  $u(t) = 0$ is a solution of \eqref{eq:invscale}. 

For time $t\geq t_*$ a continuous extension of $p$ is given by the constant $p(t)=p(t_*)$ and $u(t)=u(t_*)$.
Due to the one-homogeneity, with $t_{*} = \lambda/\gamma$ we obtain 
$$p(t_{*}) \in \partial J(\ul) = \partial J(\gamma \ul), \qquad K^*(f-Ku(t_*)) = 0 \, \text{.}$$ 
Thus, $\partial_t p = 0$ yields indeed a solution of \eqref{eq:invscale} for $t \geq t_*$.
\end{proof}
\end{thm}

\subsection{Noisy Data}
As in Section \ref{sec:exactrec} the case of noisy data is a bit more complicated. In order to recover a singular vector exactly despite the contamination of the data $f$ with noise, we basically need the signal ratio $\mu$ as introduced in Theorem \ref{thm:eigfctnoisy} to equal the singular value $\lambda$. We also mention that the stopping time is the regularization parameter in inverse scale space methods, thus we can only expect the exact reconstruction to happen in a time interval $(t_*,t_{**})$.  More precisely we obtain the following result:
\begin{thm}\label{thm:invscalspeignoisy}
Let $J$ be one-homogeneous and let $\ul$ be a
singular vector with corresponding singular value $\lambda$. The data $f$ is assumed to be corrupted by noise $n$, i.e. $f = \gamma K\ul + n$ for a positive constant $\gamma$, such that there exist positive constants $\mu$ and $\eta$ that satisfy \eqref{eq:noisyeigcond} and \eqref{eq:snrcond}. Then, a solution of the inverse scale space flow \eqref{eq:invscale} is given by
\begin{equation}
	u(t) = \left\{ \begin{array}{ll} 0 & \text{if } t < t_* \\ c \ul & \text{if } t_* \leq t < t_{**}  \end{array} \right. 
\end{equation}
for 
\begin{align}
c = \gamma + \frac{\lambda - \mu}{\eta} \, \text{.}\label{eq:isscondnoisy}
\end{align}
and 
$$t_{*} = \frac{\lambda \eta}{\lambda + \gamma \eta - \mu} < t_{**} = \eta.  $$
\end{thm}
\begin{proof}
With a similar argumentation as in the proof of Theorem \ref{thm:invscalspeigclean} we obtain $u(t) = 0$ for $t < t_*$ and
$$ p(t_*) = t_* \gamma K^*K\ul + t_* K^* n $$
as the corresponding subgradient to the first non-zero $u$ for a critical time $t_*$. Analogous to the proof of Theorem \ref{thm:eigfctnoisy} we can treat the relation above as a convex combination of $\lambda K^{*}K$ and \eqref{eq:noisyeigcond} for any $\beta \in [0, 1]$, and determine
$ \beta = \frac{\lambda}{\lambda + \gamma \eta - \mu}$ and subsequently $t_*$ as above.
Moreover, we see that  $u(t_{*}) = c \ul$ is a feasible solution with subgradient $p(t_*)$, which we extrapolate as constant for further times up to some time $t_{**}.$ Then,
from $p(t_{*}) = t_{*} K^{*}f$ and $\partial_t p(t) = K^{*}(f - cK\ul)$ we conclude
\begin{align*}
p(t) &= t K^{*}f - (t - t_{*}) c K^{*}K\ul\\
&= t K^{*}(\gamma K\ul + n) - (t - t_{*}) c K^{*}K\ul\\
&= t K^{*}n + (\gamma t - c (t - t_{*})) K^{*}K\ul \, \text{.}
\end{align*}
To obtain $p(t) \in \partial J(u(t_*))$ we again compare convex combinations of \eqref{eq:noisyeigcond} and $\lambda K^*K \ul$ with parameter $\beta \in [0,1]$. We need to choose $\beta=\frac{t}\eta$, which is only possible for $t < t_{**} = \eta$. Further we obtain that
$\mu = \gamma t - c (t - t_{*})$ needs to hold. This identity can be verified with the above formulas for $t_{*}$ and $c$. 
\end{proof}

%

\section{Further Examples}\label{sec:exm}


In the following we shall discuss several further examples to illustrate the use of nonlinear singular values:

\subsection{Hilbert Space Norms}

The obvious first starting point is to consider regularizations with Hilbert space norms, i.e. 
\begin{equation}
	J(u)=\norm{u}_{\umc} \, \text{.}
\end{equation}
Note that we focus on the one-homogeneous case here, i.e. we do not use the squared Hilbert space norm as in the standard formulation of Tikhonov regularization. However, it is easy to check that there is a one-to-one relation between the regularization parameters of the problems with squared and non-squared norms, such that they are equivalent. It is straight-forward to see that the singular values are determined from
\begin{equation}
	\lambda K^* K u = \frac{u}{\norm{u}_{\umc}} \, \text{,} 
\end{equation}
and with our normalization we see $\lambda = \norm{u}_{\umc}$.
Thus if $u_n$ is a singular vector in the classical definition
\begin{equation}
	K^* K u_n = \sigma_n^2 u_n, 
\end{equation}
it is also a singular vector in the new definition (at least after appropriate normalization). The linear singular values $\sigma_n$ are related to the novel values $\lambda_n$ via
\begin{equation}
	\sigma_n^2 = \frac{1}{\lambda_n \norm{u_n}_{\umc}} = \frac{1}{\lambda_n^2},  
\end{equation}
which is consistent with our original definition as singular values being related to the regularization functional rather than to $K$. 

The fact that singular values yield exact solutions of the variational problem is not new and is directly inferred as a special case of the standard theory (cf. \cite{EHN96}). However, it is surprising how the behavior of the inverse scale space method changes when rescaling from the squared norm to the one-homogeneous case. In the case of $J(u) = \frac{1}{2}\|u\|^2_{\umc}$ the inverse scale space method is equivalent to Showalter's method (cf. \cite{showalter})
$$ \partial_t u = K^*(f-Ku), $$
and it is well-known that singular vectors follow an exponential dynamic, i.e. if $f=K\ul$ for a singular vector $\ul$, then 
$$u(t) = (1-e^{-t/\lambda}) \ul$$
in case of $K = I$. The behavior changes completely in the one-homogeneous case as we can conclude from the results in the previous section, since the solution remains zero in finite time and then jumps exactly to $u(t) = \ul$ at the critical time. 

\subsection{Total Variation}

We have already used the ROF model for denoising at several instances as a simple illustrative example, in particular in spatial dimension one. 
In this section we want to extend the considerations of the one dimensional ROF model to data that is corrupted by noise. Moreover, we want to highlight the connection between singular vectors and characteristic functions of so-called calibrable sets as introduced in \cite{alter1, alter2}, with respect to the isotropic total variation functional in higher spatial dimension. In \cite{caselles} the theory of calibrable sets has also been extended to more general (and in particular anisotropic) regularization functionals, which we disregard here for the sake of simplicity. Instead, we introduce an analytical solution of the anisotropic total variation regularization in terms of the singular vector definition , similar to the previous examples in spatial dimension one.

\subsubsection*{Unbiased Recovery in Practice}
In this section we briefly want to illustrate that in practice the violation of
the assumptions of the Theorems \ref{thm:invscalspeigclean} and
\ref{thm:invscalspeignoisy} can indeed yield undesired artifacts in the
reconstruction. We therefore want to investigate the piece-wise constant function $u_4:[0, 1] \rightarrow \{-1, 1\}$
\begin{align}
u_4(x) := \begin{cases} 1 & x \in \left[\frac{1}{4}, \frac{3}{4}\right]\\ -1 &
\text{else} \end{cases} \,
\text{.}\label{eq:cylinder}
\end{align}
It is straightforward to see that $u_4$ is a singular vector of $\tv$ with
singular value $\lambda = 4$, and with the corresponding dual singular vector $p_4$ satisfying the relation $p_4 = q_4^\prime$ (in terms of a weak derivative) for $q_4$ defined as
\begin{align}
q_4(x) := 4 \begin{cases} -x & x \in \left[0, \frac{1}{4}\right[ \\ x -
\frac{1}{2} & x \in \left[\frac{1}{4}, \frac{3}{4}\right]\\ 1 - x & x \in
\left]\frac{3}{4}, 1\right] \end{cases} \,
\text{.}\label{eq:dualcylinder}
\end{align}
Both functions are visualized in Figure \ref{fig:unbrecpractice1}.
\begin{figure}[!ht]
\begin{center}
\includegraphics[scale=0.4]{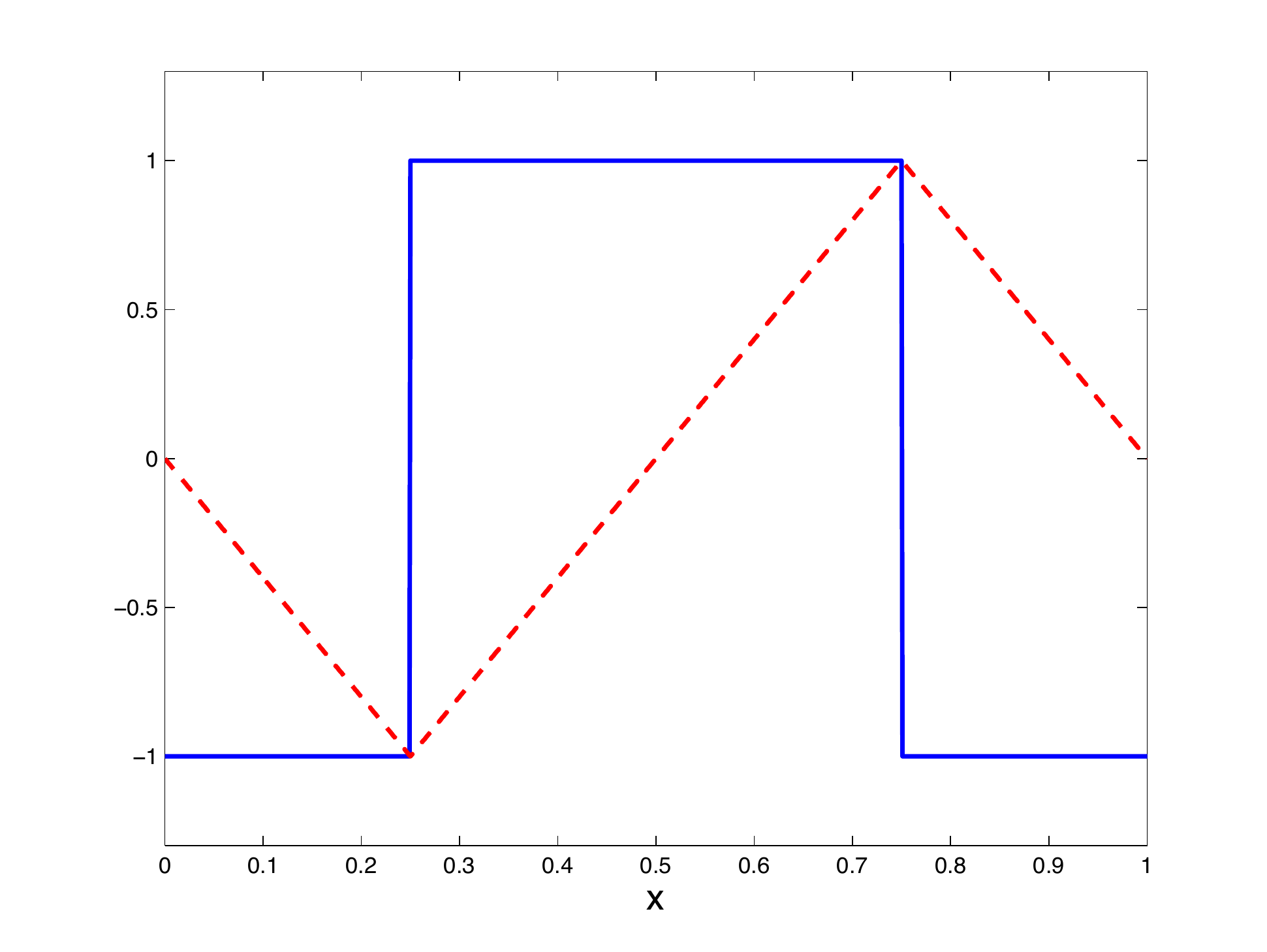}
\end{center}
\caption{The functions $u_4$ (solid blue line) and $q_4$ (dashed red line) as defined in \eqref{eq:cylinder} and \eqref{eq:dualcylinder}, respectively.}
\label{fig:unbrecpractice1}
\end{figure}
Now assume we want to compute the minimizer of \eqref{eq:rof} for our data given
in terms of $f(x) = u_4(x) + n(x)$. Here, $n$ represents a noise function which
we assume to have mean zero (i.e. $\int_0^1 n(x) ~dx = 0$), and to fulfill $N(0)
= N(1)$, with $N$ denoting the primitive of $n$ (i.e. $N^\prime(x) = n(x)$). In
order to apply Theorem \ref{thm:eigfctnoisy} we need to
guarantee the existence of constants $\mu > 0$ and $\eta \geq 1/\alpha$ such that
\eqref{eq:noisyeigcond} is satisfied. We therefore make the attempt to define
\begin{align*}
q(x) := \frac{\mu}{4} q_4(x) + \eta N(x) \, \text{,} 
\end{align*}
for which we obtain $q(0) = q(1) = 0$, due to the definition of $n$. Moreover,
we discover
\begin{align*}
\dupr[{L^2([0, 1])}]{q^\prime}{u_4} &= \mu + 2 \eta 
\int_{\frac{1}{4}}^{\frac{3}{4}} n(x) ~dx \nonumber\\
&= 4 \left( \frac{\mu}{4} + \frac{\eta}{2} \left( N\left(\frac{3}{4}\right) -
N\left(\frac{1}{4}\right) \right)\right) \, \text{,}
\end{align*} 
which equals $\tv(u_4) = 4$ if $\mu$ satisfies
\begin{align*}
\mu = 4 - 2\eta\left( N\left(\frac{3}{4}\right) - N\left(\frac{1}{4}\right) \,
\right) \text{.}
\end{align*} 
Assume $N(3/4) \geq N(1/4)$, then we even obtain
\begin{align}
\mu \leq 4 - \frac{2}{\alpha} \left( N\left(\frac{3}{4}\right) -
N\left(\frac{1}{4}\right) \right) \leq 4 = \lambda \, \text{,}\label{eq:mucond}
\end{align}
due to $\eta \geq 1/\alpha$.
Note that in order to obtain $\mu > 0$ we need to ensure
\begin{align}
\alpha > \frac{1}{2}\left( N\left(\frac{3}{4}\right) -
N\left(\frac{1}{4}\right) \right) \, \text{,}\label{eq:alphabound}
\end{align}
otherwise Theorem \ref{thm:eigfctnoisy} cannot be applied. 
Assuming to choose $\eta = 1/\alpha$, the loss of contrast modifies to
\begin{align*}
c = 1 - 4 \alpha + 2\left(N\left(\frac{3}{4}\right) -
N\left(\frac{1}{4}\right)\right)
\end{align*}
in case that \eqref{eq:alphabound} does hold.

\begin{figure}[!ht]
\begin{center}
\subfigure[$\alpha = \frac{19\pi + 2}{152\pi}$]{\includegraphics[width=0.32\textwidth]{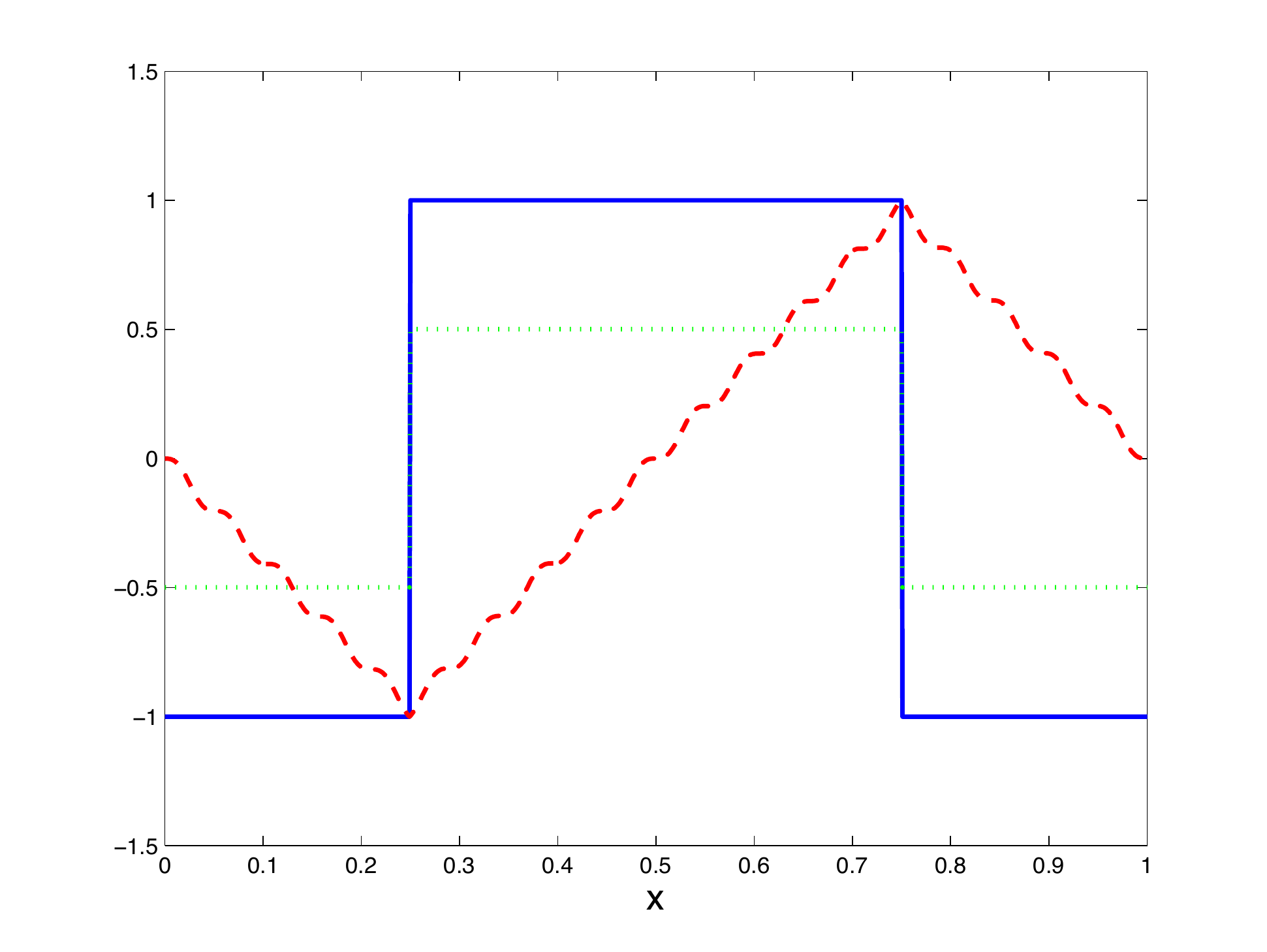}}
\subfigure[$\alpha = \frac{1}{75\pi}$]{\includegraphics[width=0.32\textwidth]{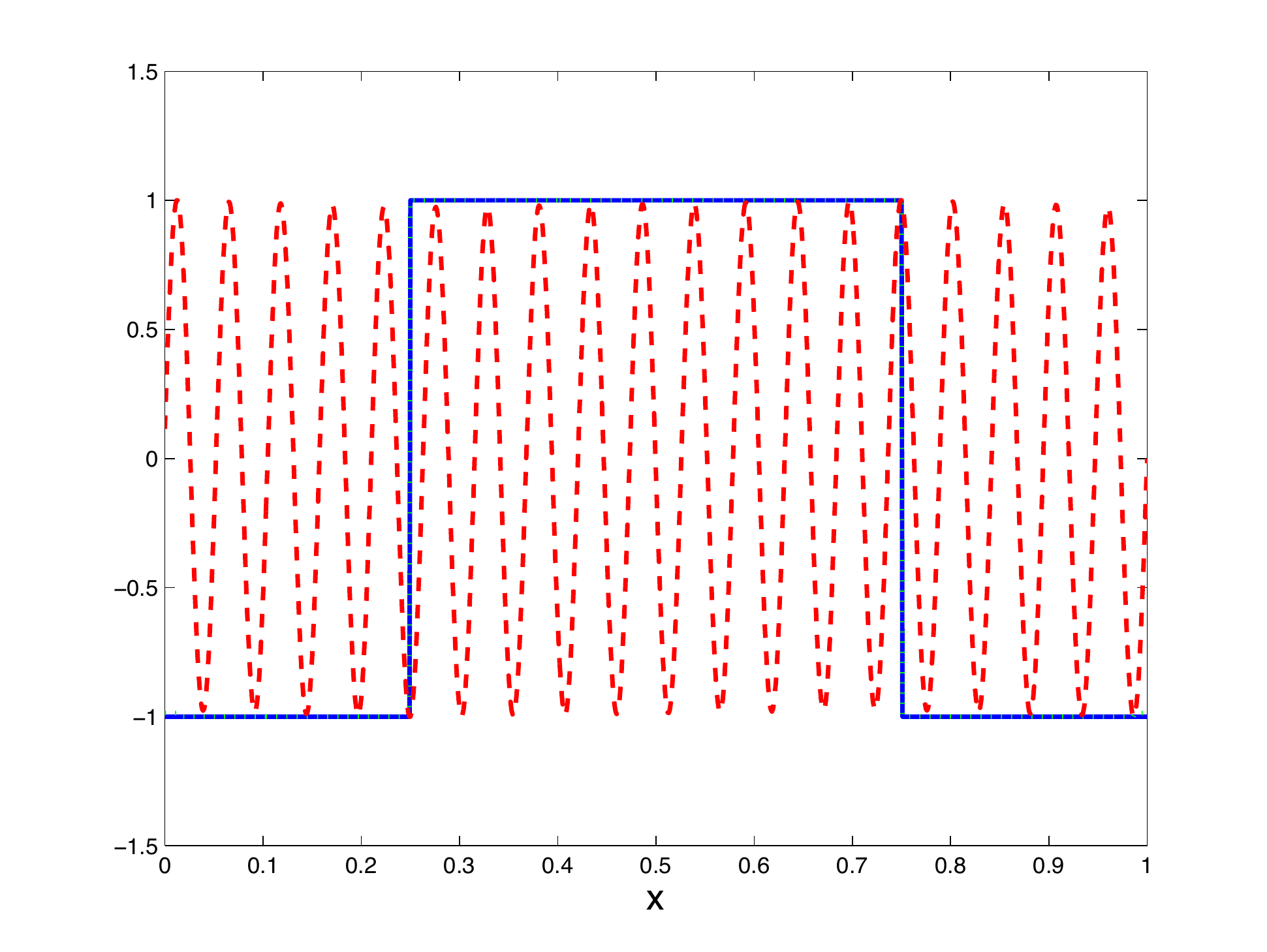}}
\subfigure[$\alpha = \frac{1}{74\pi}$]{\includegraphics[width=0.32\textwidth]{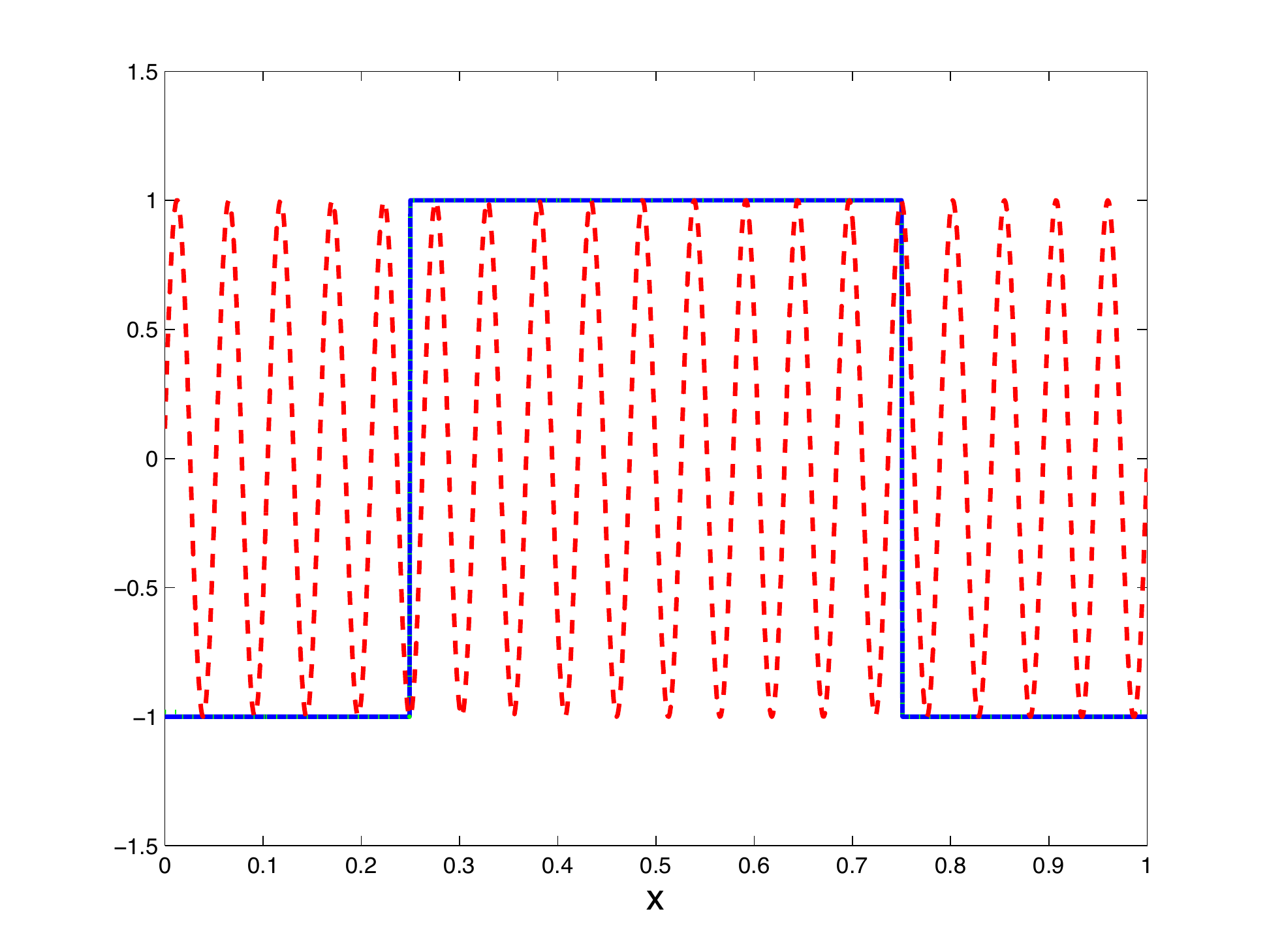}}\\
\subfigure[Closeup, $\alpha = \frac{1}{75\pi}$]{\includegraphics[width=0.49\textwidth]{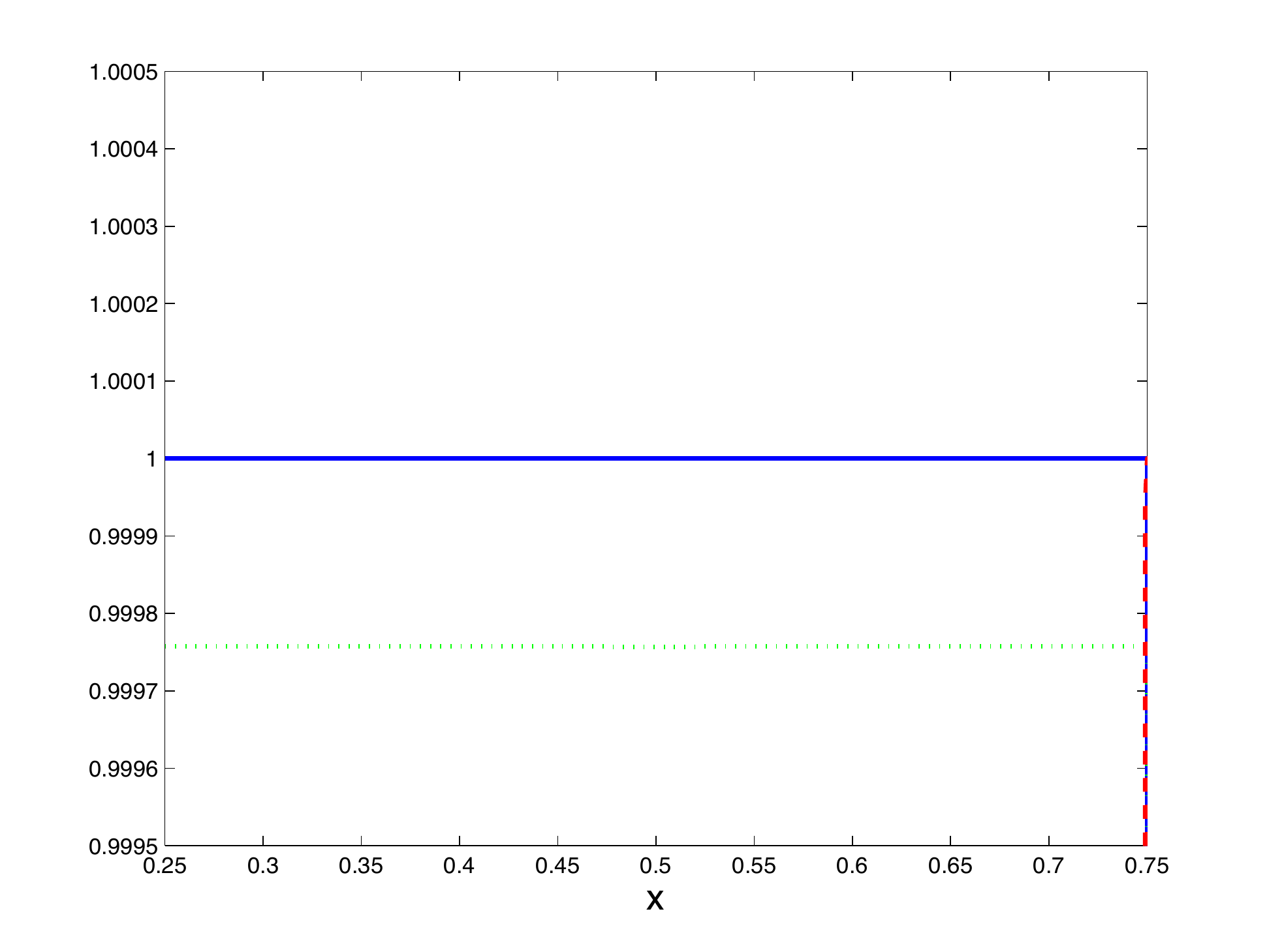}}
\subfigure[Closeup, $\alpha = \frac{1}{74\pi}$]{\includegraphics[width=0.49\textwidth]{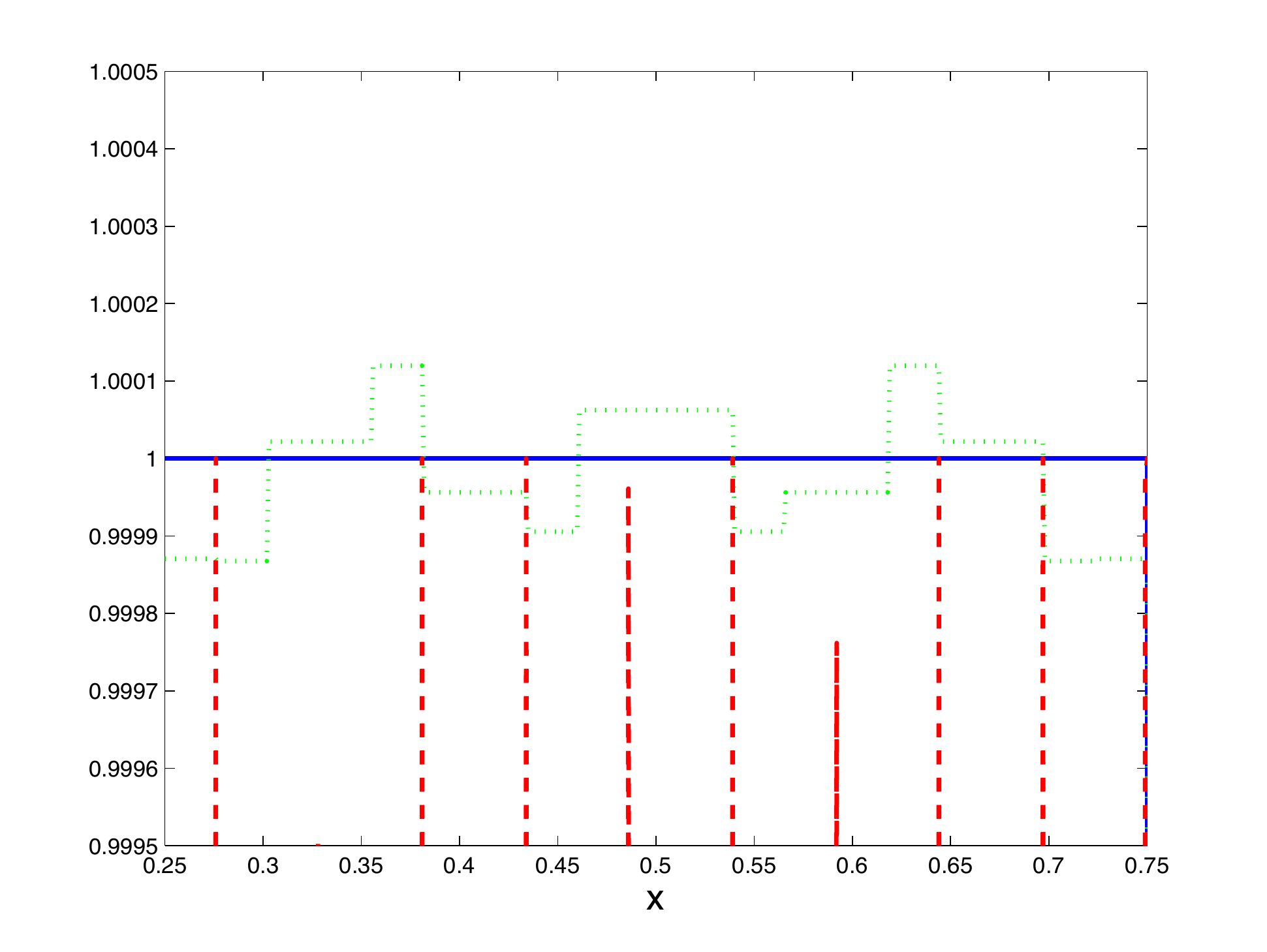}}
\end{center}
\caption{Computational ROF-reconstructions for input datum $f = u_4 + n$, with $u_4$ and $n$ being defined as in \eqref{eq:cylinder} and \eqref{eq:cylindernoise}, respectively. It is remarkable to see that as soon as $\alpha$ is chosen such that Theorem \ref{thm:eigfctnoisy} cannot be applied, the numerical computations fail to compute a multiple of $u_4$, indicating the sharpness of the Theorem.}
\label{fig:unbrecpractice2}
\end{figure}

\noindent Let us consider a specific example now. We decide to choose the periodic
function
\begin{align}
n(x) := A \cos\left(38 \pi x\right)\label{eq:cylindernoise}
\end{align}
to be our noise function, with amplitude $A$ and frequency $38$. Note that this
noise function satisfies the mean zero property as well as $N(0) = N(1)$, for
$N(x) = (A \sin\left(38 \pi x\right))/(38 \pi)$. Moreover, we compute $N(3/4) =
A/(38\pi)$ and $N(1/4) = -A/(38\pi)$, so that \eqref{eq:mucond} now reads as
$\mu \leq 4 \left( 1 - A/(38 \alpha \pi)\right)$, and according to \eqref{eq:alphabound}, $\alpha$ should be chosen to satisfy 
\begin{align*}
\alpha > \frac{A}{38\pi} \, \text{.}
\end{align*}
In Figure \ref{fig:unbrecpractice2} you can see several computational solutions of \eqref{eq:rof} for the specified input datum $f = u_4 + n$, for $A = 1/2$ and numerous $\alpha$-values. The computations nicely indicate that as soon as the assumptions of Theorem \ref{thm:eigfctnoisy} are violated, artifacts are introduced in the computational reconstruction.

In the unbiased case of Theorem \ref{thm:invscalspeignoisy} we may conclude from the considerations above that we obtain $u(t) = u_4$ as the solution of the Inverse Scale Space Flow \eqref{eq:invscale}, for $4 \leq t <  (38\pi)/A$ and $A < (19\pi)/2$.

\subsubsection*{The ROF Model and Calibrable Sets}

In \cite{meyer} Y. Meyer has basically proved that the characteristic function of a circle is a singular vector of isotropic total variation on $\mathbb{R}^N$. In \cite{alter1, alter2} the class of characteristic functions that correspond to singular vectors in the terminology of this paper has been extended to calibrable sets. In the following we want to recall properties of calibrable sets and show why they correspond to singular vectors of total variation.
\begin{defi}
Let $C \subset \mathbb{R}^2$ be a bounded, convex and connected set with its boundary $\partial C$ being of class $C^{1, 1}$. Then $C$ is called \normalfont calibrable \itshape if there exists a vector field $\xi \in L^\infty(C; \mathbb{R}^2)$ with $\| \xi \|_\infty \leq 1$ such that
\begin{align}
\begin{split}
-\div ~\xi &= \text{const} = \lambda_C := \frac{P(C)}{|C|} \ \text{in $C$}\\
\xi \cdot \nu &= - 1 \ \text{a. e. on $\partial C$}
\end{split}
\end{align}
holds, for $\nu$ denoting the outer unit normal to $\partial C$, $P(C)$ denoting the perimeter and with $|C|$ representing the volume of $C$.
\end{defi}
\begin{rem}\normalfont
Note that the perimeter simply equals the isotropic total variation $\tv(\chi_C)$ of the corresponding characteristic function $\chi_C$.
\end{rem}
\begin{thm}
In \cite{alter2} it has been proved that for calibrable sets the following conditions are satisfied.
\begin{itemize}
	\item $C$ is the solution of the problem
	\begin{align*}
	\min_{X \subset C} P(X) - \lambda |X| \, \text{.}
	\end{align*}
	\item The inequality 
	\begin{align*}
		\esup_{x \in \partial C} \kappa(x) \leq \lambda
	\end{align*}
	holds, with $\kappa(x)$ denoting the curvature of $\partial C$ at point $x$.	
\end{itemize}
\end{thm}
In \cite{bellettini} it has already been proved that for calibrable sets $C$ the solution of the isotropic total variation with an characteristic function $\chi_{C}$ as the input datum satisfies $\hat{u}(x) = (1 - \lambda \alpha) \chi_C (x)$. Thus, Theorem \ref{thm:eigfctclean} implies that the input datum already has to be a singular vector with singular value $\lambda$.

We finally mention that Agueh and Carlier \cite{aguehcarlier} have investigated a related class of ground state problems for total variation with constraints of the form $\int_\Omega G(|u|)~dx = 1$. Such may be interesting for denoising with noise models different from additive Gaussian, where hardly any examples of exact solutions exist, except for single ones in \cite{benning}.

\subsubsection*{The Anisotropic ROF Model in Two Dimensions}
Analytical solutions of the isotropic ROF model have widely been studied and
discussed in literature. However, analytical solutions of the anisotropic ROF
model have not attracted a similar attention, although many of them are easier
to describe, as we are going to see in the following.

\begin{figure}[!ht]
\begin{center}
\subfigure[On-top view of $u_{\sqrt{32}}$]{\includegraphics[width=0.49\textwidth]{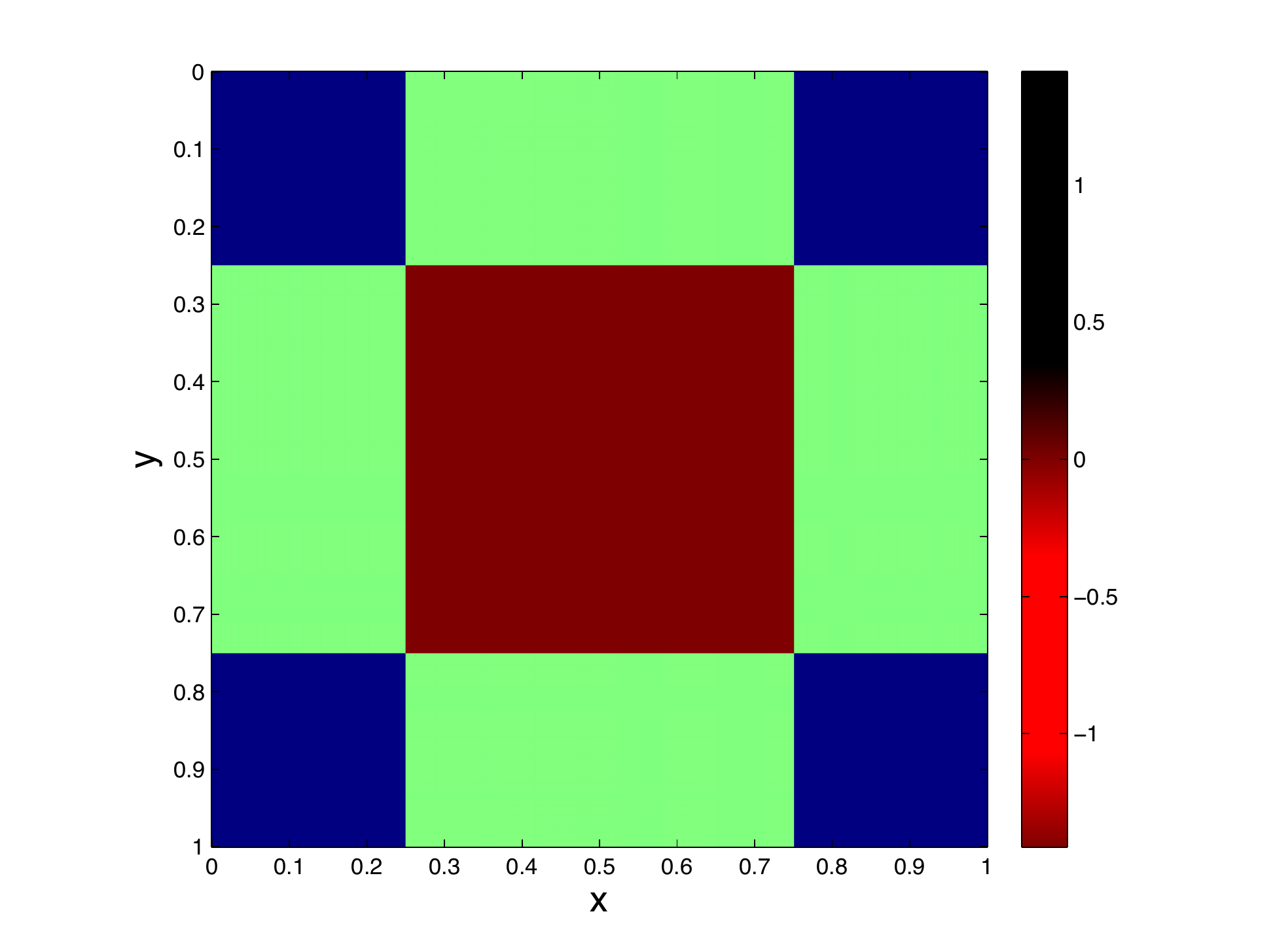}}
\subfigure[3D view of $u_{\sqrt{32}}$]{\includegraphics[width=0.49\textwidth]{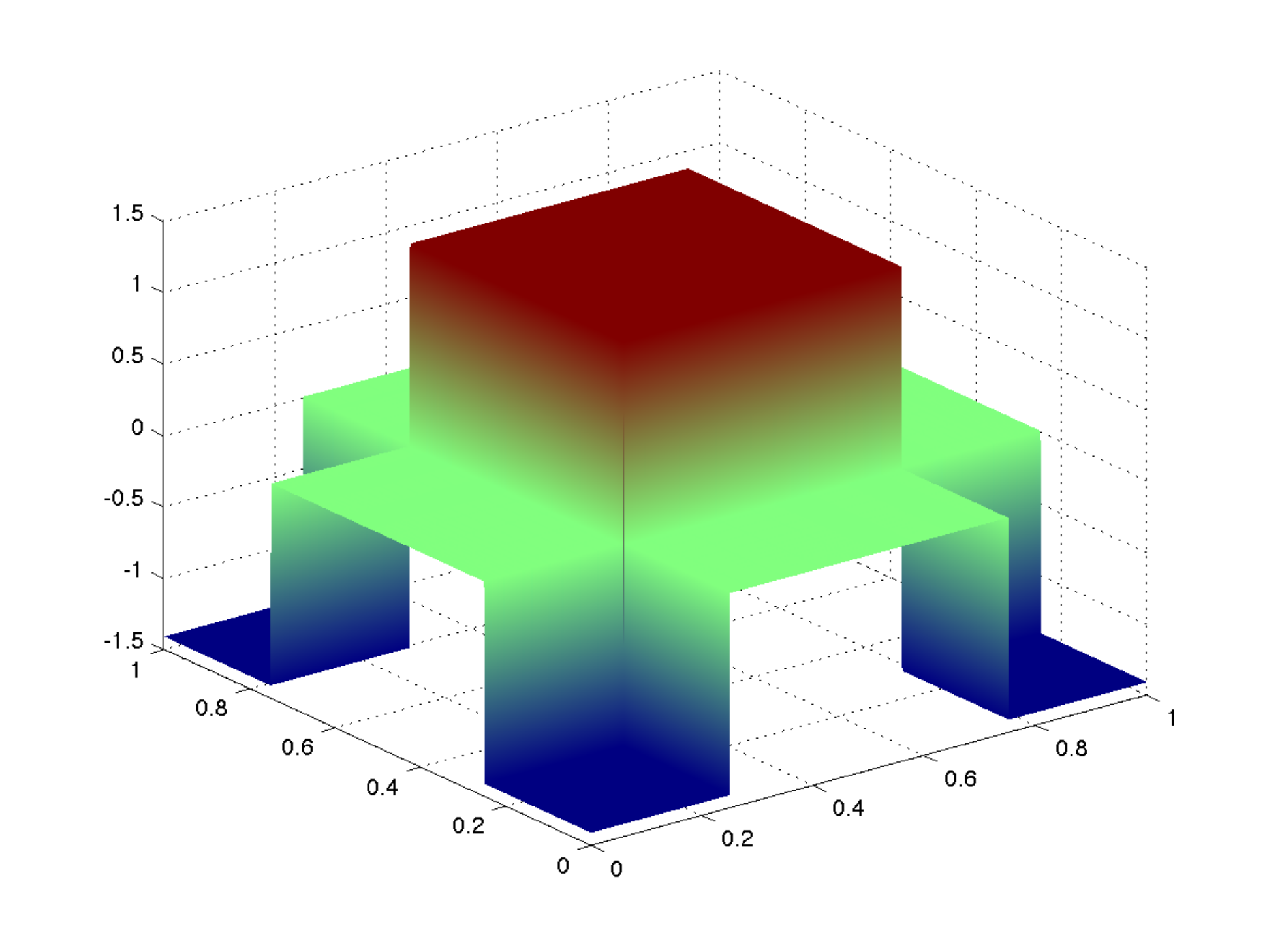}}\\
\subfigure[3D views of $q_{\sqrt{32}}^x$ and $q_{\sqrt{32}}^y$]{\includegraphics[width=0.7\textwidth]{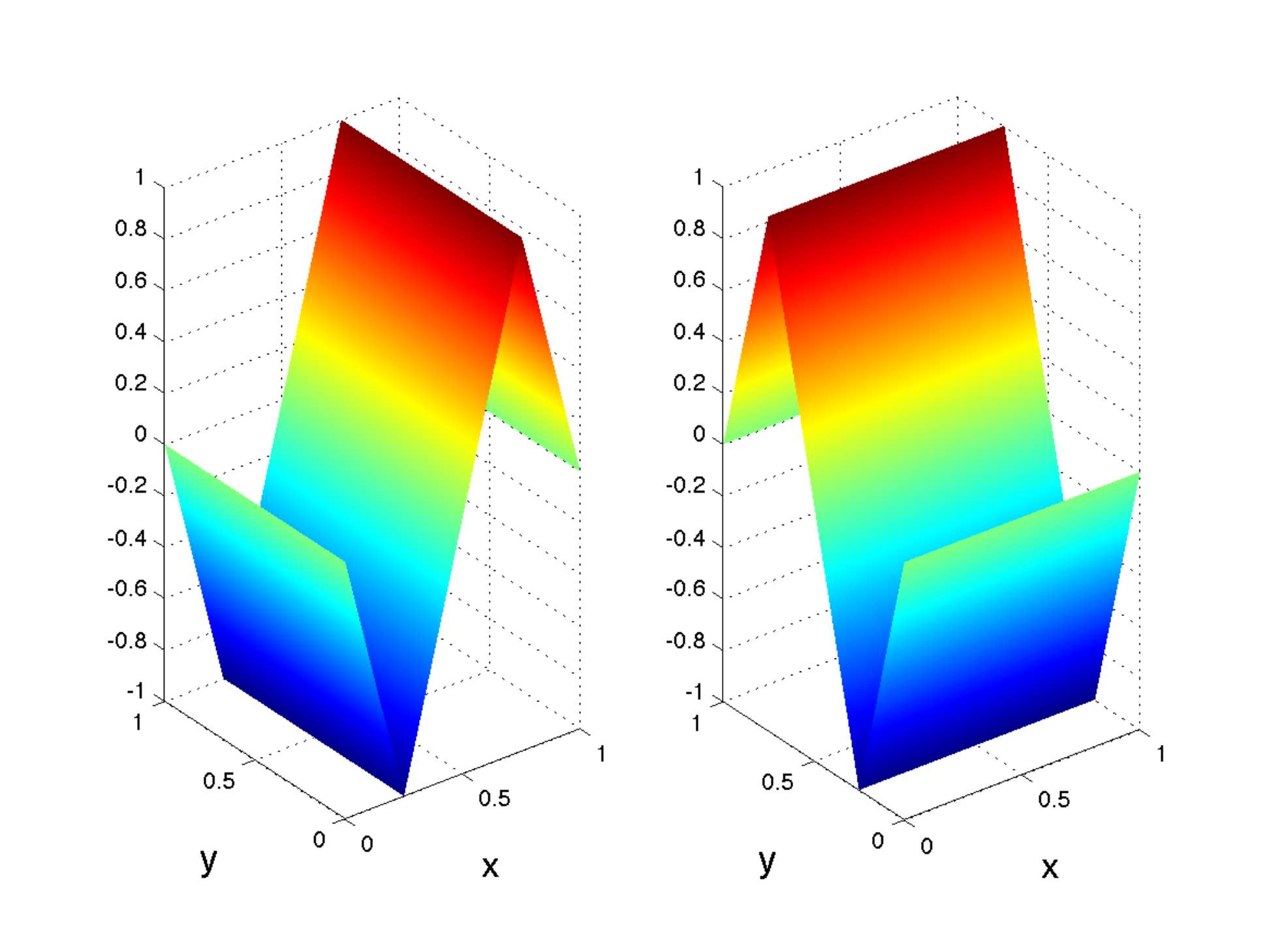}}
\end{center}
\caption{On-top and three-dimensional view of the singular vector $u_{\sqrt{32}}$ and its dual variables $p_{\sqrt{32}}^x$ and $p_{\sqrt{32}}^y$.}
\label{fig:anisotvsv}
\end{figure}

Recall \eqref{eq:cylinder} and its dual singular vector $p_4 = q^\prime_4$, with $q_4$ being defined in \eqref{eq:dualcylinder}.
Let us define the two-dimensional functions $q_{\sqrt{32}}^x:[0, 1]^2
\rightarrow [-1, 1]$ and $q_{\sqrt{32}}^y:[0, 1]^2 \rightarrow
[-1, 1]$ with $q_{\sqrt{32}}^x(x, y) := q_4(x)$ and $q_{\sqrt{32}}^y := q_4(y)$.
Then it is easy to see that the weak divergence of $q_{\sqrt{32}} = (q_{\sqrt{32}}^x,
q_{\sqrt{32}}^y)$ divided by $\sqrt{32}$ yields the function
\begin{align}
u_{\sqrt{32}}(x, y) = \sqrt{2} \begin{cases} 1 & (x, y) \in
\left[\frac{1}{4}, \frac{3}{4}\right]\\ 0 & \left(\left(\left|x -
\frac{1}{2}\right| \leq \frac{1}{4}\right) \wedge \left(\left|y -
\frac{1}{2}\right| > \frac{1}{4}\right)\right) \vee \left(\left(\left|x -
\frac{1}{2}\right| > \frac{1}{4}\right) \wedge \left(\left|y -
\frac{1}{2}\right| \leq \frac{1}{4}\right)\right)\\ -1 & \text{else}\end{cases}
\, \text{.}
\end{align}
Thus, with the same considerations as before we are able to prove that
$u_{\sqrt{32}}$ is a singular vector of $\tv$ with singular value $\lambda =
\sqrt{32}$, since we observe
\begin{itemize}
  \item $q_{\sqrt{32}}^x n_x = 0$ and $q_{\sqrt{32}}^y n_y = 0$, with $n_x$ and
  $n_y$ denoting the outer unit normals of $q_{\sqrt{32}}^x$ and
  $q_{\sqrt{32}}^y$ in $x$- and in $y$-direction, respectively
  \item $\|q \|_{L^\infty([0, 1]^2; \mathbb{R}^2)} =
  \max\left(\left|q_{\sqrt{32}}^x\right|, \left|q_{\sqrt{32}}^y\right|\right) =
  1$
  \item $\dupr[L^2({[0, 1]^2}; \mathbb{R}^2)]{\div
  q_{\sqrt{32}}}{u_{\sqrt{32}}} = \tv(u_{\sqrt{32}})$
\end{itemize}
The singular vector $u_{\sqrt{32}}$ and the dual vectors $q_{\sqrt{32}}^x$ and $q_{\sqrt{32}}^y$ are visualized in Figure \ref{fig:anisotvsv}.

\subsubsection*{Denoising Vector Fields}

Another generalization of the one-dimensional ROF model to multiple dimensions has been discussed recently in \cite{briani}, namely the denoising of vector fields $f \in L^2(\Omega;\R^n)$ via minimizing
\begin{equation}
	u = \argmin_{u \in L^2(\Omega;\R^n)} \left\{\frac{1}2 \int_\Omega (f-u)^2~dx + \int_\Omega |\div  ~u |~dx \right\}\, \text{.}
\end{equation}
As in the case of total variation, the $L^1$-norm of the divergence has to be generalized to a weak form
\begin{equation}
	J(u) = \sup_{\substack{\varphi \in C_0^\infty(\Omega) \\ \|
\varphi \|_{L^\infty(\Omega)} \leq 1}} \int_\Omega u \cdot \nabla \varphi ~dx.
\end{equation}

Concerning our investigation of ground states and singular values, this model yields an example with a huge set of trivial ground states. Any function $u \in L^2(\Omega;\R^n)$ such that $\nabla \cdot u = 0$ is obviously a ground state. In order to compute a nontrivial ground state we obtain the condition 
$$ \int_\Omega u  \cdot v~dx = 0 \qquad \text{for all $v$ with $\nabla \cdot v = 0$,} $$
and that $u$ has to be a gradient field, i.e.
$$ \lambda u = \nabla q \, \text{.} $$
The scalar $q$ is obtained from the minimization of 
$$\int_\Omega |\Delta q|~dx \quad \text{ subject to } \quad 
\int_\Omega \|\nabla q\|_{\ell^2(\R^n)}^2~dx = 1 \, \text{.} $$



\subsection{Support Pursuit}

While sparsity regularization with discrete $\ell^1$-functionals has been studied extensively in the last decade, the continuum analog was investigated only recently. At a first glance it seems that $L^1(\Omega)$ would be the straight-forward extension in terms of function spaces, but similar to the case of total variation the lack of a weak-star topology in $L^1(\Omega)$ prevents the applicability and often the existence of minimizers. Again the solution is to extend the variational approach to a slightly larger space, which is a dual space. In this case this dual space is the space of Radon measures ${\cal M}(\Omega)$, which is the dual space of $C_0(\Omega)$. The appropriate regularization functional as introduced in \cite{bredieshanna} is the zero-order total variation of a measure $\mu \in {\cal M}(\Omega)$, i.e.
\begin{equation}
	J(\mu) =  \sup_{\substack{\varphi \in C_0(\Omega) \\ \|
\varphi \|_{\infty} \leq 1}} \int_\Omega  \varphi ~d\mu.  
\end{equation}
The setup used in \cite{bredieshanna} is to choose $K=L^*$, with $L:\hmc \rightarrow C_0(\Omega)$ being a bounded linear operator. This allows to avoid working with the complicated dual space of ${\cal M}(\Omega)$ for some arguments. We note that in this setup regular singular vectors can rather be obtained from 
\begin{equation}
 \lambda LL^* \mu \in \partial J(\mu) \, \text{.} 	
\end{equation}
In \cite{decastro} this problem was analyzed in a compressed sensing setting, when $K$ consists of a finite number of forward projections (in a polynomial basis), so that $L$ can be written down explicitly. 

\subsubsection*{Bounded Operators}

Here we will consider two related cases including most practical examples we are aware of (except the Radon integral transform): First of all we analyze integral operators $K_\infty:{\cal M}(\Omega) \rightarrow L^2(\Sigma)$  of the form
\begin{equation}
	(K_\infty u)(x) = \int_\Omega k(x,y) ~d\mu(y)
\end{equation}
with a continuous kernel $k$. Then we shall turn to projection operators with $M$ measurements $K_M: {\cal M}(\Omega) \rightarrow \R^M$ of the form
\begin{equation}
	(K_M u) = \left(\int_\Omega k_j(x)~d\mu(x)\right)_{j=1,\ldots,M} \, \text{,}
\end{equation}
assuming again the $k_j$ to be continuous. Notice that $K_M$ can also be thought of as a semidiscretization of the operator $K_\infty$, e.g. by collocation methods, so it is natural to compare at least ground states for those operators. 
 
Our aim is to verify a natural extension of the $\ell^1$-case, where the ground state is a vector with a single nonzero entry. The natural analogue in the space of Radon measures is a concentrated measure $\delta_x$ for $x \in \Omega$ (we use the notation $\delta_x$ due to the relation to the Dirac delta distribution), with
\begin{equation}
	\int_\Omega \varphi(y)~d\delta_x = \varphi(x) \qquad \forall \varphi \in C_0(\Omega) \, \text{.}
\end{equation}
Indeed we can show that ground states are of this form and give a reasonably simple condition on their location.

\begin{thm} \label{thm:support}
Let $K_\infty$ and $K_M$ be as above. Then
\begin{itemize}
\item  A ground state $\mu_0^\infty$ of $K_\infty$ is given by $\mu_0^\infty=c\delta_z$, with $z$ satisfying 
$$ \int_\Omega k(x,z)^2 ~dx \geq \int_\Omega k(x,y)^2 ~dx \qquad \forall~y \in \Omega \, \text{,} $$
and with $c$ fulfilling
$$ c= \lambda_0^\infty = \frac{1}{\sqrt{\int_\Omega k(x,z)^2 ~dx}} \, \text{.} $$

\item A ground state $\mu_0^M$ of $K_M$ is given by $\mu_0^M=c \delta_z$, with $z$ satisfying 
$$ \sum_j k_j(z)^2 \geq \sum_j k_j(y)^2 \qquad \forall~y \in \Omega \, \text{,} $$
so that $c$ fulfills
$$ c = \lambda_0^\infty = \frac{1}{\sqrt{\sum_j k_j(z)^2}} \, \text{.} $$
\end{itemize}
\end{thm} 
\begin{proof}
We prove a slightly more general version: for an operator $K:{\cal M}(\Omega) \rightarrow \hmc$, $\mu_0=c \delta_z$ is a ground state if 
$$ f(y) := \Vert K \delta_y \Vert_{\hmc} $$
attains a maximum at $z$. This yields the statements of the theorem as special cases.

\noindent Define $c=\lambda_0 = \frac{1}{\Vert K \delta_z \Vert_{\hmc}}$. Then we see that 
$$ \langle \lambda K^* K \mu_0, \mu_0 \rangle_{{\cal M}(\Omega)} = c = J(\mu_0). $$
Moreover, for any $\mu \in {\cal M}(\Omega)$ we have
$$ \langle \lambda K^* K \mu_0, \mu \rangle_{{\cal M}(\Omega)} = \lambda \langle K \delta_z, K \mu \rangle_{\hmc} 
\leq \lambda \Vert K \mu \Vert_{\hmc} \, \text{.}
$$ 
We thus verify that $\lambda^2 \Vert K \mu \Vert_{\hmc}^2 \leq J(\mu)^2$, or equivalently
$$ \Vert K \mu \Vert_{\hmc}^2 \leq J(\mu)^2 \Vert K\delta_z \Vert_{\hmc}^2 $$ 
holds. First of all we discover
$$ \Vert K \mu \Vert_{\hmc}^2 = \langle \mu, K^*K \mu \rangle_{{\cal M}(\Omega)} \leq J(\mu) \Vert K^* K \mu \Vert_{\infty} \, \text{.} $$
For $y \in \Omega$ we further estimate
$$ \vert (K^* K \mu)(y) \vert = \vert \langle K \delta_y, K \mu \rangle_{\hmc} \vert \leq \Vert K \delta_y \Vert_{\hmc} ~\Vert K\mu \Vert_{\hmc} \, \text{,} $$
and thus obtain
$$  \Vert K^* K \mu \Vert_{\infty} = \sup_y \vert (K^* K \mu)(y) \vert \leq {\sup_y f(y)} \Vert K\mu \Vert_{\hmc} \, \text{,}$$ 
which finally yields the assertion.
\end{proof}

\subsubsection*{An Unbounded Operator}

As a simple sketch of the ROF model in three dimensions (when embedding of $BV$ into $L^2$ fails), we study the case of $K^* K$ being the inverse Laplacian, i.e. $Ku=v$, with $v \in H_0^1(\Omega)$ solving $-\Delta v = u$ if $u \in {\cal M}(\Omega) \cap H^{-1}(\Omega)$, and for $\Omega$ denoting the unit ball in $\R^3$. Since a Dirac delta distribution is not an element of $H^{-1}(\Omega)$ in spatial dimension three, we cannot obtain such measures as singular vectors, hence the latter can at most be concentrated on manifolds with higher dimension. 

\noindent The equation for the singular vector becomes 
$$ \lambda \mu =  - \Delta p, \qquad p \in \partial J(\mu)$$
and we can look for radially symmetric solutions $\mu=M(r)$, $p=P(r)$. Hence
$$ \lambda r^2 M = - \partial_r (r^2 \partial_r P) $$
with the additional condition $P(1)=0$. 
Now a canoncial measure concentrated on a codimension one manifold is the one on a sphere with radius $R \in (0,1)$, which corresponds to $M$ being a concentrated measure in $r=R$, i.e. $M = c \delta_R$ with $c$ to be determined from the normalization condition 
$$ \int_0^1 P(r)^2 r^2~dr = 1 \, \text{.} $$  Then $P$ can be computed as
\begin{equation}
	P(r) = \left\{ \begin{array}{ll} \lambda c \frac{R^2}r(1-r) & \text{if } r \geq R \\ \lambda c R(1-R) & \text{if } r < R\end{array}\right. 
\end{equation}
Now $P$ attains a maximum at $r=R$, and we need to choose $\lambda$ such that $P(R)=1$ holds, which yields 
$\lambda = \frac{1}{cR(1-R)}$. Hence, we conclude that a measure concentrated on a sphere of radius $R$ is a singular vector, the smallest singular value in this class is obtained for $R=\frac{1}2$. Note that $\lambda \rightarrow \infty$ for $R\rightarrow 	1$ or $R\rightarrow 0$, which confirms that neither a concentrated measure in the origin nor a measure concentrated on $\partial \Omega$ is a singular vector.

\subsection{Sparsity and Variants}

As mentioned above sparsity-enforcing regularizations played an important role in image analysis and inverse problems in the last years. 
The usual setup in $\ell^1$-Regularization is $\umc=\ell^1(\R^n)$ and $\hmc=\ell^2(\R^m)$, the forward operator can thus be identified with a matrix $K \in \R^{m \times n}$. The proof of Theorem \ref{thm:support} immediately implies that a ground state is given by $\gamma e_i$, where $i \in \{1,\ldots,n\}$ is the index of a row with maximal Euclidean norm, i.e. $\Vert Ke_i \Vert_2 \geq \Vert Ke_j \Vert_2$ for any  $j \in \{1,\ldots,n\}$.
In the following we discuss two other relevant examples related to sparsity and their ground states, respectively singular vectors. 

\subsubsection*{Low Rank}

In order to compute matrix-valued solutions of low rank, the nuclear norm has been considered in various papers recently, respectively shown to be an exact relaxation of minimal rank problems in some cases (cf. \cite{fazel,recht}). In this case $\umc=\R^{m \times n}$ and $\hmc=\R^M$ with 
\begin{equation}
	J(u) = \sum_{j=1}^{\min\{m,n\}} \sigma_j(u) \, \text{,}
\end{equation}
for which $\sigma_j$ are the (classical) singular values of the matrix $u$, and with $M \ll (m \times n)$ denoting the number of known entries of $u$.

It seems natural that a ground state is of rank one. To see this, we take an arbitrary matrix $u$ with singular value decomposition
$$ u = \sum_{j=1}^{\min\{m,n\}} \sigma_j U_j V_j^T \, \text{.} $$
Then we have
$$ \Vert Ku \Vert_{\hmc} = \left\Vert \sum_{j=1}^{\min\{m,n\}} \sigma_j K U_j V_j^T \right\Vert_{\hmc} \leq \sum_{j=1}^{\min\{m,n\}} \sigma_j \Vert K U_j V_j^T \Vert_{\hmc}
\leq \max{\Vert K U_j V_j^T \Vert_{\hmc}}  J(u). $$
Equality is obtained if $u$ is a rank-one matrix. Thus, we conclude that the ground state is of rank one and obviously it is a multiple of $UV^T$, where $U$ and $V$ maximize $\Vert K U V^T \Vert_{\hmc}$ among all orthogonal matrices.

\subsubsection*{Joint Sparsity}

In some applications it is more reasonable that few groups of variables have nonzero entries instead of just a few single variables being nonzero. This is modeled by so-called joint sparsity or group lasso approaches (cf. \cite{wakin,yuanlin,meier}), the most prominent example being
\begin{equation}
	J(u) = \sum_{i=1}^n \sqrt{\sum_{j=1}^{m_i} u_{ij}^2},
\end{equation}
in $\umc=\R^N$ and $\hmc = \R^M$, where $N=\sum_{i=1}^n m_i. $ 
Since the goal is to obtain group sparsity, we expect solutions such that $u_{i\cdot}$ vanishes for most indices $i$. In particular one expects the ground state such that $u_{i\cdot}$ is a nonzero vector for only a single index $i$.

In order to characterize the ground state we introduce the matrices $A_i \in \R^{M \times m_i}$ representing the linear operator $K$ restricted to elements $u$ supported in the index set $\{i\} \times \{1,\ldots,n_i\}$. We shall also use the notation $U_i=(u_{ij})_{j=1,\ldots,m_i} \in \R^{n_i}$. With those we can write
$$J(u)  = \sum_{i=1}^n \Vert U_i \Vert_{\ell^2(\R^{m_i})}, \quad \Vert K u \Vert_{\ell^2(\R^M)} = \left\Vert \sum_{i=1}^n A_i U_i \right\Vert_{\ell^2(\R^M)} \, \text{.} $$ 
By the triangle inequality 
\begin{equation}
	\Vert K u \Vert_{\ell^2(\R^M)} \leq \sum_{i=1}^n  \Vert A_i U_i \Vert_{\ell^2(\R^M)} \leq \sum_{i=1}^n \sigma_i^{\text{max}} \Vert U_i \Vert_{\ell^2(\R^{m_i})}
	\leq J(u) \max_{i \in \{1, \ldots, n\}}  \sigma_i^{\text{max}} \, \text{,}
\end{equation}
for which $\sigma_i^{\text{max}}$ is the largest singular value of $A_i$. Equality is obtained if $U_k$ is the singular vector corresponding to singular value $\sigma_k^{\text{max}}$ with
$$  \sigma_k^{\text{max}} \geq  \sigma_i^{\text{max}}  \qquad \forall i \neq k$$ 
and all other $U_i$ equal zero. Thus, there are ground states with only one row different from zero, which perfectly corresponds to the motivation of group sparsity. The pattern of the nonzero row is a classical singular vector of the restricted matrix $A_i$.

%

\subsection{Infimal-Convolution Regularization}

Due to deficiencies of standard regularization functionals, constructions like infimal convolution of multiple regularization functionals have been considered recently (cf. 
\cite{chamlion,aujgillblcham,aujcham,starck,infconv,bredies}). The infimal convolution (inf-convolution) of two convex functionals $J_1$ and $J_2$ is defined as
\begin{equation}
	J(u):= \inf_{v \in \umc} (J_1(v) + J_2(u-v)),
\end{equation}
and appears to be a good way to combine the advantages of different regularization functionals. Ideally one would hope that the inf-convolution of $J_1$ and $J_2$ can lead to exact reconstruction of all solutions that are reconstructed exactly with $J_1$ or $J_2$.

\noindent Since the related variational problem can be reformulated as 
\begin{equation}
	(v, w) = \argmin_{v, w} \left\{\frac{1}2 \Vert K(v+w) - f \Vert^2 + \alpha (J_1(v) + J_2(w))\right\} \, \text{,}
\end{equation}
we can directly consider singular vectors in the product space for $u=(v,w)$, which are characterized by
\begin{equation}
	\lambda K^*K(v+w) = p_1 = p_2, \quad p_1 \in \partial J_1(v), \ p_2 \in \partial J_2(w) \, \text{.}
\end{equation}
In general we cannot expect that singular vectors of $J_1$ or $J_2$ are again singular vectors of the inf-convolution. The simplest case would be a singular vector $v$ 
$$ \lambda K^*K v = p_1 \in \partial J_1(v), \qquad w = 0. $$
Then we need that $p_1 \in \partial J_2(0)$, which is difficult to achieve for general combinations. However, the construction works at least for the ground state of one-homogeneous functionals $J_1$ and $J_2$. Let $v_0$ be the ground state of $J_1$ and $w_0$ be the one of $J_2$. Moreover we assume that $J_1(v_0) \leq J_2(w_0)$. 
Then we can estimate
$$ J_1(v) + J_2(w) \geq J_1(v_0) \Vert Kv \Vert + J_2(w_0) \Vert Kw \Vert \geq J_1(v_0) \Vert K(v+w)\Vert. $$ 
Equality is achieved if $w=0$ and $v=v_0$, hence the ground state of $J_1$ is also a ground state of $J$. Note that for $J_2(w_0) > J_1(v_0)$ we may conclude that $w_0$ is not a ground state, potentially not even a singular vector. Since such inequalities depend on the scaling of $J_1$ and $J_2$ this suggests that one should use a scaling such that the smallest singular values are equal.

\section{Conclusions and Open Problems}

In this paper we have generalized the notion of singular values and singular vectors to nonlinear regularization methods in Banach spaces and demonstrated their usefulness in the analysis. In particular we have derived results on the bias of variational methods and scale estimates, which have a particular geometric interpretation in the case of total variation denoising. Moreover, we have shown that singular vectors are the solutions that can be reconstructed exactly (up to a multiplicative constant) by variational regularization techniques. 

A major open problem is to obtain a constructive approach for computing singular values and singular vectors, or at least ground states, of arbitrary problems either analytically or numerically. A computational approach for similar problems was already discussed in \cite{heinbuehler}, as well as for similar problems with quadratic constraints in \cite{edelman,laiosher,manton,wen,yamada}. Our computational experiments indicate that such approaches can indeed compute singular vectors, however they do not converge robustly to the ground state and it is difficult to control to which singular vector the method will converge. 

\section{Acknowledgments}
This work was supported by the German Science Foundation DFG through grant BU 2327/6-1.
The authors want to thank Michael M\"{o}ller for fruitful discussions and for hints on the the proof of Lemma \ref{lem:tv1Dgs}.


\end{document}